\documentclass[final]{colt2012}

\usepackage{epsf}
\usepackage{graphicx}
\usepackage{amsmath}
\usepackage{amssymb}
\usepackage{algorithmic}
\usepackage{algorithm}
\newcommand{\U}{\mathcal U}
\newcommand{\B}{\mathcal B}
\newcommand{\V}{\mathbb V}
\renewcommand{\P}{\mathbb P}
\newcommand{\al}{\alpha}
\newcommand{\no}{ \epsilon}

\newlength{\minipagewidth}
\newcommand{\bookbox}[1]{\small
\par\medskip\noindent
\framebox[\columnwidth]{
\begin{minipage}{\minipagewidth} {#1} \end{minipage} } \par\medskip }

\newcommand{\beq}{\begin{equation}}
\newcommand{\eeq}{\end{equation}}

\newcommand{\beqa}{\begin{eqnarray}}
\newcommand{\eeqa}{\end{eqnarray}}

\newcommand{\beqan}{\begin{eqnarray*}}
\newcommand{\eeqan}{\end{eqnarray*}}

\newcommand{\E}{\mathbb{E}}

\newcommand{\alg}{\mathcal A}

\newcommand{\hmu}{\hat{\mu}}

\newcommand{\var}{\sigma^2}
\newcommand{\hvar}{\hat{\sigma}^2}

\newcommand{\si}{\sigma}
\newcommand{\hsi}{\hat{\sigma}}

\newcommand{\N}{\mathcal N}

\newcommand{\ind}[1]{\mathbb I\left\lbrace {#1} \right\rbrace}

\newcommand{\R}{\mathbb R}

\newtheorem{assumption}{Assumption}

\title[Minimax Number of Strata for Online Stratified Sampling given Noisy Samples]{Minimax Number of Strata for Online Stratified Sampling given Noisy Samples}
 \coltauthor{\Name{Alexandra Carpentier} \Email{alexandra.carpentier@inria.fr}\\
 \addr INRIA Lille Nord-Europe, 40, avenue Halley, 59650 Villeneuve d'Ascq, France
 \AND
 \Name{R\'emi Munos} \Email{remi.munos@inria.fr}\\
 \addr INRIA Lille Nord-Europe, 40, avenue Halley, 59650 Villeneuve d'Ascq, France
 }

\begin{document}

\maketitle

\begin{abstract}
We consider the problem of online stratified sampling for Monte Carlo integration of a function given a finite budget of $n$ noisy evaluations to the function. More precisely we focus on the problem of choosing the number of strata $K$ as a function of the budget $n$.
%We exhibit a trade-off between what we define as \textit{quality of a stratification}, which measures the performance of the optimal oracle strategy on the stratification and improves with $K$, and the \textit{pseudo-regret} of an adaptive strategy, which measures the capacity of an algorithm to adapt to the optimal strategy on a give partition and deteriorates with $K$.
We provide asymptotic and finite-time results on how an oracle that has access to the function would choose the partition optimally. %We also provide an improved pseudo-regret bound  of order $\tilde O(K^{1/3}n^{-4/3})$ for the MC-UCB algorithm introduced in \citep{MC-UCB}.
In addition we prove a \textit{lower bound} on the learning rate for the problem of stratified Monte-Carlo. As a result, we are able to state, by improving the bound on its performance, that algorithm MC-UCB, defined in~\citep{MC-UCB}, is minimax optimal both in terms of the number of samples n and the number of strata K, up to a $\sqrt{\log(nK)}$. This enables to deduce a minimax optimal bound on the difference between the performance of the estimate outputted by MC-UCB, and the performance of the estimate outputted by the best oracle static strategy, on the class of H\"older continuous functions, and upt to a $\sqrt{\log(n)}$.
\end{abstract}

\begin{keywords}
Online learning, stratified sampling, Monte Carlo integration, regret bounds.
\end{keywords}

%\title{Stratified Monte-Carlo for integrating a noisy function}

%\author{Alexandra, Mohammad, Alessandro, R\'emi, Peter}
%\author{Alexandra Carpentier, \textbf{Alessandro Lazaric,} \textbf{Mohammad Ghavamzadeh,} \textbf{R{\'e}mi Munos} \\
%INRIA Lille - Nord Europe, Team SequeL, France\\
%\\
%\textbf{Peter Auer}\\
%University of Leoben, Franz-Josef-Strasse 18, 8700 Leoben, Austria.
%Alexandra Carpentier, \\
%\textbf{Alessandro Lazaric,} \\
%\textbf{Mohammad Ghavamzadeh,} \\
%\textbf{R{\'e}mi Munos} \\
%SequeL Project, \\
%INRIA Lille - Nord Europe, France \\
%{alexandra.carpentier@inria.fr}\\
%{alessandro.lazaric@inria.fr}\\
%{mohammad.ghavamzadeh@inria.fr}\\
%{remi.munos@inria.fr}
%\And
%Peter Auer \\
%University of Leoben, \\
%Franz-Josef-Strasse 18 \\
%8700 Leoben, Austria\\
%{auer@unileoben.ac.at}
%}

%\maketitle

%\vspace{0.2in}

%%%%%%%%%%%%%%%%%%%%%%%%%%%%%%%%%%%%%%%%%%%%%%%%%%%%%%%%%%%%%%%%%%%%%%%%%%%%%%%
%% ABSTRACT
%%%%%%%%%%%%%%%%%%%%%%%%%%%%%%%%%%%%%%%%%%%%%%%%%%%%%%%%%%%%%%%%%%%%%%%%%%%%%%%

\section{Introduction}

The objective of this paper is to provide an efficient strategy for Monte-Carlo integration of a function $f$ over a domain $[0,1]^d$. We assume that we can query  the function $n$ times. Querying the function at a time $t$ and at a point $x_t \in [0,1]^d$ provides a noisy sample 
\begin{equation}\label{eq:model}
\boldsymbol{f(x_t)+s(x_t) \no_{t}},
\end{equation}
where $\epsilon_t$ is an independent sample drawn from $\nu_{x_t}$. Here $\nu_x$ is a distribution with mean 0, variance 1 and whose shape may depend on x\footnote{It is the usual model for functions in heterocedastic noise. We isolate the standard deviation on a point $x$, $s(x)$, in the expression of the noise, since this quantity is very relevant.}. This model is actually very general (see Section~\ref{s:setting}).

%When doing this, we stop considering the direction along which we are not stratifying, and we project\footnote{What we mean by project is that we choose the first $d$ sample randomly, uni} the function from the space of large dimension to the space of smaller dimension where we stratify. 

%Crude Monte-Carlo consists in sampling randomly, uniformly over the domain. The mean of the noisy sample is an unbiased estimate of the integral of the function.
Stratified sampling is a well-known strategy to reduce the variance of the estimate of the integral of $f$, when compared to the variance of the estimate provided by crude Monte-Carlo. The principle is to partition the domain in $K$ subsets called \textit{strata} and then to sample in each stratum (see \citep{rubinstein2008simulation}[Subsection 5.5] or \citep{glasserman2004monte}). %Fixing the number of strata and the number of samples in each stratum reduces the amount of randomness in the estimate.
If the variances of the strata are known, there exists an optimal static allocation strategy which allocates the samples proportionally to the measure of the stratum times their standard deviation (see Equation~\ref{BestStrat} in this paper for a reminder). We refer to this allocation as optimal oracle strategy for a given partition. In the case that the variations of $f$ and the standard deviation of the noise $s$ are unknown, it is not possible to adopt this strategy.

Consider first that the partition of the space is fixed. A way around this problem is to estimate the variations of the function and the amount of noise on the function in the strata \textit{online} (exploration) while allocating the samples according to the estimated optimal oracle proportions (exploitation). This setting is considered in~\citep{A-EtoJou10,Grover,MC-UCB}. In the long version~\citep{rapp-tech-MC-UCB} of the last paper, the authors propose the so-called MC-UCB algorithm which is based on Upper-Confidence-Bounds (UCB) on the standard deviation. They provide upper bounds for the difference between the mean-squared error\footnote{The mean squared error is measured with respect to the quantity of interest, i.e.~the integral of $f$.} of the estimate provided by MC-UCB and the mean-squared error of the estimate provided by the optimal oracle strategy (optimal oracle variance). The algorithm performs almost as well as the optimal oracle strategy.
However, the authors of~\citep{rapp-tech-MC-UCB} do not infirm nor assess the optimality of their algorithm with a lower bound as benchmark. As a matter of fact, no lower bound on the rate of convergence (to the oracle optimal strategy) for the problem of stratified Monte-Carlo exists, to the best of our knowledge. %We emphasize the importance of lower-bounds on learning rates, since lower bounds serve to characterize the hardness of a problem and to have a benchmark on the efficiency of algorithms.
Still in the same paper~\citep{rapp-tech-MC-UCB}, the authors do not at all discuss on how to \textit{stratify} the space. In particular, they do not pose the problem of what an \textit{optimal oracle} partition of the space is, and do not try to answer on whether it is possible or not to attain it.

The next step is thus to efficiently design the partition. There are some interesting papers on that topic such that~\citep{glasserman1999asymptotically,kawai2010asymptotically,A-EtoForJouMou11}. The recent, state of the art, work of \cite{A-EtoForJouMou11} describes a strategy that samples \textit{asymptotically} almost as efficiently as the optimal oracle strategy, and at the same time adapts the direction and number of the strata online. This is a very difficult problem. The authors do not provide proofs of convergence of their algorithm. However for static allocation of the samples, they present some properties of the stratified estimate when the number of strata goes to infinity and provide convergence results under the optimal oracle strategy. As a corollary, they prove that the more strata there are, the smallest the optimal oracle variance. %An interesting property they find is that choosing the right \textit{number} of strata is the key-point whereas the exact position of their boundaries is less important, since an adaptive algorithm would select the \textit{number of points} in each stratum to the specific shape (and thus variance) of the stratum. 

\paragraph{Contributions:}

%This paper is concerned with the choice of the number of strata. 

%Indeed, the asymptotic analysis of \cite{A-EtoForJouMou11}, although highlighting the fact that the number of strata has to go to infinity but ``not too fast'' (in an asymptotic sense), does not provide more precise results on \textit{how exactly} this number has to be chosen in order to maximize the \textit{speed} at which the variance of the algorithm converges to the smallest possible variance. We call this rate pseudo-regret\footnote{Defined in Section~\ref{s:setting}.}.

%This limitation is inherent to their analysis. In order to output precise calibrations for the number of strata, it is essential to have a finite-time analysis for (a) the distance between the optimal allocation when the number of strata goes to infinity and the optimal allocation for a fixed number of groups and (b) the rate at which a given algorithm approximates the optimal allocation. As \cite{A-EtoForJouMou11} chose to have very mild assumptions and to perform an asymptotic analysis, they can not provide such results.

The more strata there are, the smaller the variance of the estimate computed when following the optimal oracle strategy. %Indeed, when the number of strata goes to infinity (and the diameter of the biggest stratum goes to $0$), the variance of the estimate of the optimal, oracle strategy converges to the smallest possible variance, defined in~\citep{A-EtoForJouMou11}.
However, the more strata there are, the more difficult it is to estimate the variance within each of these strata, and thus the more difficult it is to perform almost as well as the optimal oracle strategy. Choosing the number of strata is thus crucial and this is the problem we address in this paper. This defines a trade-off similar to the one in model selection (and in all its variants, e.g.~density estimation, regression...): The wider the class of models considered, i.e.~the larger the number of strata, the smaller the distance between the true model and the best model of the class, i.e.~the approximation error. But the larger the estimation error.
\\
Paper~\citep{A-EtoForJouMou11}, although proposing no finite time bounds, develops very interesting ideas for bounding the first term, i.e.~the approximation error. As pointed out in paper e.g.~\citep{MC-UCB}, it is possible to build algorithms that have a small estimation error. By constructing tight and finite-time bounds for the approximation error, it is thus possible to propose a number of strata that minimizes an upper bound on the performance. It is however not clear how consistent this choice is, i.e.~how much it can be improved. The essential ingredients for choosing efficiently a partition are thus lower bounds \textit{on the estimation error, and on the approximation error}.

The objective of this paper is to propose a method for choosing the minimax-optimal number of strata. Our contributions are the following. %In order to do that, we derive some results that we also believe to be interesting in themselves.
\begin{itemize}
\item We first present results on what we call the \textit{quality} $Q_{n,\N}$ of a given partition in $K$ strata $\N$ (i.e., using the previous analogy to model selection, this would represent the approximation error). %Quality is defined as the difference between the variance of the estimate computed according to the optimal, oracle strategy for a fixed stratification, and the variance of the estimate computed according to the optimal, oracle strategy on the best partition.
Using very mild assumptions we compute a lower bound on the variance of the estimate given by the optimal oracle strategy on the optimal oracle partition. Then if the function and the standard deviation of the noise are $\alpha-$H\"older, and also if the strata satsfy some assumptions, we prove that $Q_{n,\N} = O(\frac{K^{\alpha/d}}{n})$. %If the strata are hypercubes and have the same side length $\frac{1}{K^{1/d}}$ and measure $\frac{1}{K}$, the quality is bounded as $Q_{n,\N_K} = O(\frac{K^{\alpha/d}}{n})$, where $\alpha$ is the H\"older exponent of the function.
This bound is also minimax optimal on the class of $\alpha-$H\"older functions.

\item We then present results on the estimation error for the estimate outputted by algorithm MC-UCB of~\citep{MC-UCB} (pseudo-regret in the terminology of~\citep{MC-UCB}). %We emphasize the benefit of a finite-time analysis for an algorithm as it provides a rate of convergence of the algorithm both in terms of the number of samples, but also in terms of the number of strata.
In this paper, we improve the analysis of the MC-UCB algorithm when compared to paper~\citep{MC-UCB} in terms of the dependence on $K$. The problem independent bound on the pseudo-regret in~\citep{MC-UCB} is of order\footnote{Here $\tilde O$ is a $O$ up to a polynomial $\log(n)$.} $\tilde O(Kn^{-4/3})$, and we tighten this bound in this paper so that it is of order $\tilde O( K^{1/3} n^{-4/3})$.

\item We provide the first \textit{lower bound} (on the pseudo-regret) for the problem of online Stratified Sampling. The bound $\Omega(K^{1/3} n^{-4/3})$ is tight and \textit{matches the upper-bound of MC-UCB both in terms of the number of strata and the number of samples}. This is the main contribution of the paper, and we believe that the proof technique for this bound is original.

\item Finally, we combine the results on the quality and on the pseudo-regret of MC-UCB to provide a value on the number of strata leading to a minimax-optimal trade-off (up to a $\sqrt{\log(n)}$) on the class of $\alpha-$H\"older functions. %Since MC-UCB is minimax optimal (up to a multiplicative constant) both in terms of the number of samples and number of strata, the pseudo-regret of MC-UCB launched with this carefully chosen number of strata attains the minimax-optimal rate (up to a multiplicative constant) on the class of H\"older functions with known exponent $\alpha$.
%Two distinct analyses corresponding to two different configurations of this problem: i) If the stratas are homogeneous or if the budget is big (i.e. if the problem is ``simple''), we are able to state that we deviate from the targetted variance from only $\widetilde O(n^{-3/2})$ while ii) if the stratas are inhomogeneous or the budget is small (i.e. the problem is difficult), then we deviate from $O(n^{-4/3})$.
\end{itemize}

The rest of the paper is organized as follows. In Section~\ref{s:setting} we formalize the problem and introduce the notations used throughout the paper. Section~\ref{s:conver} states the results on the quality of a partition. Section~\ref{s:algo} improves the analysis of the MC-UCB algorithm, and establishes the lower bound on the pseudo-regret.
Section~\ref{s:tradeoff} reports the best trade-off to choose the number of strata. And in Section~\ref{s:experiments}, we illustrate how important it is to choose carefully the number of strata. We finally conclude the paper and suggest future works.

\section{Setting}\label{s:setting}

We consider the problem of numerical integration of a function $f: [0,1]^d \rightarrow \R$ with respect to the uniform (Lebesgue) measure. We dispose of a budget of $n$ queries (samples) to the function, and we can allocate this budget \textit{sequentially}. When querying the function at a time $t$ and at a point $x_t$, we receive a noisy sample $X(t)$ of the form described in Equation~\ref{eq:model}.%= f(x) + s(x)\no_{t}}$ where $\boldsymbol{\no_{t} \sim \nu_x}$ is an independent noise of mean $0$ and variance $1$, whose distribution might vary with $x$, and $s(x)\geq 0$ is the heterocedastic (varying with $x$) variance of the noise on the function. %Note that for any time $t$, $\int_{[0,1]^d} f(x) dx = \int_{[0,1]^d} \E_{\no_x}[X_x(t)]dx$ .

We now assume that the space is stratified in $K$ Lebesgue measurable strata that form a partition $\N$. We index these strata, called $\Omega_{k}$, with indexes $k\in\{1,\ldots,K\}$, and write $w_k$ their measure, according to the Lebesgue measure. We write
 $\mu_{k} = \frac{1}{w_{k}} \int_{\Omega_{k}} \E_{\no \sim \nu_x} [f(x) + s(x)\no] dx = \frac{1}{w_{k}} \int_{\Omega_{k}} f(x) dx$ their mean and $\si_{k}^2 = \frac{1}{w_{k}} \int_{\Omega_{k}} \E_{\no \sim \nu_x} [(f(x) + s(x)\no - \mu_{k})^2] dx$ their variance. These mean and variance correspond to the mean and variance of the random variable $X(t)$ when the coordinate $x$ at which the noisy evaluation of $f$ is observed is chosen uniformly at random on the stratum $\Omega_k$.% $U \sim \mathcal U_{\Omega_k}$, and $\mathcal U_{\Omega_k}$ is the Lebesgue measure restricted to stratum $\Omega_k$.

We denote by $\alg$ an algorithm that allocates online the budget by selecting at each time step $1\leq t\leq n$ the index $k_t\in\{1,\dots ,K\}$ of a stratum and then sampling uniformly the corresponding stratum $\Omega_{k_t}$. The objective is to return the best possible estimate $\hat \mu_{n}$ of the integral of the function $f$. We write $T_{k,n} = \sum_{t \leq n} \ind{k_t = k}$ the number of samples in stratum $\Omega_k$ up to time $n$. We denote by $\big(X_{k,t}\big)_{1\leq k \leq K, 1\leq t \leq T_{k,n}}$ the samples in stratum $\Omega_k$, and we define $\hmu_{k,n} = \frac{1}{T_{k,n}} \sum_{t=1}^{T_{k,n}}X_{k,t}$ the empirical means. We estimate the integral of $f$ by $\hmu_n = \sum_{k=1}^K w_k \hmu_{k,n}$.

If we allocate a deterministic number of samples $T_{k}$ to each stratum $\Omega_{k}$ and if the samples are independent and chosen uniformly on each stratum $\Omega_{k}$, we have
\begin{equation*}
 \E(\hmu_{n}) = \sum_{k \leq K} w_k \mu_{k} = \sum_{k \leq K} \int_{\Omega_{k}} f(u) du = \int_{[0,1]^d} f(u) du = \mu,
\end{equation*}
and also
\begin{equation*}
 \V(\hmu_n) = \sum_{k \leq K} \frac{w_k^2 \var_{k}}{T_{k}},
\end{equation*}
where the expectation and the variance are computed according to all the samples that the algorithm collected.

For a given algorithm $\alg$ allocating $T_{k,n}$ samples drawn \textit{uniformly} within stratum $\Omega_k$, we denote by \emph{pseudo-risk} the quantity 
\begin{equation}\label{risk}
 L_{n,\N}(\alg) = \sum_{k \leq K} \frac{w_k^2 \var_{k}}{T_{k,n}}.
\end{equation}

Note that if an algorithm $\alg^*$ has access the variances $\var_{k}$ of the strata, it can choose to allocate the budget in order to minimize the pseudo-risk, i.e., sample each stratum $T_{k}^* = \frac{w_k \si_{k}}{\sum_{i \leq K} w_i \si_{i}}n$ times (this is the so-called oracle allocation). These optimal numbers of samples can be non-integer values, in which case the proposed optimal allocation is not realizable. But we still use it as a benchmark. The pseudo-risk for this algorithm (which is also the variance of the estimate here since the sampling strategy is deterministic) is then
\begin{equation}\label{BestStrat}
 L_{n,\N}(\alg^*) = \frac{\Big(\sum_{k \leq K} w_k \si_{k}\Big)^2}{n} = \frac{\Sigma_{\N}^2}{n},
\end{equation}
where $\Sigma_{\N} = \sum_{k \leq K} w_k \si_{k}$. We also refer in the sequel as optimal proportion to $\lambda_k = \frac{w_k \si_{k}}{\sum_{i \leq K} w_i \si_{i}}$, and to optimal oracle strategy to this allocation strategy. Although, as already mentioned, the optimal allocations (and thus the optimal pseudo-risk) might not be realizable, it is still very useful in providing a lower-bound. No static (even oracle) algorithm has a pseudo-regret lower than $L_{n,\N}(\alg^*)$ on partition $\N$.
%Note that it is straightforward to show that the more refined the partition of the space, the smaller $L_{n,\N}(\alg^*)$, and as $\forall K$, $L_{n,\N}(\alg^*) \geq 0$, it follows that $\lim_{K \rightarrow \infty} L_{n,\N}(\alg^*) = \min_{\N measurable} L_{n,K}(\alg^*)$.

It is straightforward to see that the more refined the partition $\N$ the smaller $L_{n,\N}(\alg^*)$.
We thus define the \textit{quality of a partition} $Q_{n,\N}$ as the difference between the variance $L_{n,\N}(\alg^*)$ of the estimate provided by the optimal oracle strategy on partition $\N$, and the infimum of the variance of the optimal oracle strategy on {\em any} partition (optimal oracle partition) (with an arbitrary number of strata):
\begin{equation}\label{eq:quality}
 Q_{n,\N} = L_{n,\N}(\alg^*) - \inf_{\N' measurable} L_{n,\N'}(\alg^*).
\end{equation}

We also define the \textit{pseudo-regret} of an algorithm $\alg$ on a given partition $\N$, the difference between its pseudo-risk and the variance of the optimal oracle strategy:
\begin{equation}\label{eq:regret}
 R_{n,\N}(\alg) = L_{n, \N}(\alg) - L_{n,\N}(\alg^*).
\end{equation}

We will assess the performance of an algorithm $\alg$ by comparing its pseudo risk to the minimum possible variance of an optimal oracle strategy on the optimal oracle partition:
\begin{equation}\label{eq:regret2}
  L_{n,\N}(\alg) - \inf_{\N' measurable} L_{n,\N'}(\alg^*) = R_{n,\N}(\alg) + Q_{n,\N}.
\end{equation}

Using the analogy of model selection mentioned in the Introduction, the quality $Q_{n,\N}$ is similar to the approximation error and the pseudo-risk $R_{n,\N}(\alg)$ to the estimation error. %The objective is to choose the good class of models, i.e.~the good $K$ (and the bigger $K$, the more complex the class of models).

%In that case, not stratifying in those directions will result in a noise (heterocedastic if all the directions are not independent, which is the case for most options) when sampling randomly in the strata. More precisely, assume that the function lies in a $d^*$ dimensional space (in the hyper-cube $[0,1]^{d^*}$), and that we stratify in the first $d < d^*$ directions. 

\paragraph{Motivation for the model $f(x) + s(x)\no_{t}$.} Assume that a learner can, at each time $t$, choose a point $x$ and collect an observation $F(x,W_t)$, where $W_t$ is an independent noise, that can however depend on $x$. It is the general model for representing evaluations of a noisy function. There are many settings where one needs to integrate accurately a noisy function without wasting too much budget, like for instance pollution survey. Set $f(x) = \E_{W_t}[F(x,W_t)]$, and $s(x) \no_t = F(x,W_t) - f(x)$. Since by definition $\no_t$ is of mean $0$ and variance $1$, we have in fact $s(x) =  \sqrt{\E_{\nu_x}[(F(x,W_t) - f(x))^2]}$ and $\no_t = \frac{F(x,W_t) - f(x)}{s(x)}$. Observing $F(x,W_t)$ is equivalent to observing $f(x) + s(x)\no_{t}$, and this implies that the model that we choose is also very general.
\\
There is also a important setting where this model is relevant, and this is for the integration of a function $F$ in high dimension $d^*$. Stratifying in dimension $d^*$ seems hopeless, since the budget $n$ has to be exponential with $d^*$ if one wants to stratify in every direction of the domain: this is the curse of dimensionality. It is necessary to reduce the dimension by choosing \textit{a small amount} of directions $(1,\ldots, d)$ that are particularly relevant, and control/stratify only in these $d$ directions\footnote{This is actually a very common technique for computing the price of options, see~\citep{glasserman2004monte}.}. Then the control/stratification is only on the first $d$ coordinates, so when sampling at at a time $t$, one chooses $x = (x_1, \ldots, x_d)$, and the other $d^* - d$ coordinates $U(t) = (U_{d+1}(t),\ldots, U_{d^*}(t))$ are uniform random variables on $[0,1]^{d^*-d}$ (without any control). When sampling in $x$ at a time $t$, we observe $F(x,U(t))$.
By writing $f(x) = \E_{U(t) \sim \mathcal U([0,1]^{d^*-d})}[F(x,U(t))]$, and $s(x) \no_t = F(x,U(t)) - f(x)$, we obtain that the model we propose is also valid in this case.

\section{The quality of a partition: Analysis of the term $Q_{n,\N}$.}\label{s:conver}

In this Section, we focus on the \textit{quality} of a partition defined in Section~\ref{s:setting}.

%the  $L_{n,K}(\alg^*) - \min_{K'} L_{n,K'}(\alg^*) = L_{n,K}(\alg^*) - \lim_{K' \rightarrow +\infty} L_{n,K'}(\alg^*)$ because of monotonous convergence. Note also that as $ L_{n,K}(\alg^*) = \frac{\Sigma_K^2}{n}$, studying $Q_{n,K}$ is equivalent to studying $\Sigma_K$.

%The object $\Sigma_{K}$ is directly influenced by the number and shape of the strata. Note that if one considers a partition $\N$ and a partition $\N'$ thiner than $\N$, then $\sum_{x \in \N} w_x \si_x \geq \sum_{x \in \N'} w_x \si_x$. There is thus some monotony, and under mild assumptions there must be some sense in considering the limit when the partition gets very thin. We state in this section some results on the limit of $\Sigma_{K}$, and also on the speed of convergence.

\paragraph{Convergence under very mild assumptions}

As mentioned out in Section~\ref{s:setting}, the more refined the partition $\N$ of the space, the smaller $L_{n,\N}(\alg^*)$, and thus $\Sigma_{\N}$. Through this monotony property, we know that $\inf_{\N} \Sigma_{\N}$ is also the limit of the $(\Sigma_{\N_p})_p$ of a sequence of partitions $(\N_p)_p$ such that the diameter of each stratum goes to $0$. We state in the following Proposition that for \textit{any} such sequence, $\lim_{p \rightarrow +\infty} \Sigma_{\N_p} = \int_{[0,1]^d} s(x) dx$. Consequently $\inf_{\N} \Sigma_{\N} = \int_{[0,1]^d} s(x) dx$.

\begin{proposition}\label{lem:absconv}
Let $(\N_p)_p = (\Omega_{k,p})_{k \in \{1,\ldots, K_p\}, p\in \{1,\ldots,+\infty\}}$ be a sequence of measurable partitions (where $K_p$ is the number of strata of partition $\N_p$) such that
\begin{itemize}
\item \emph{AS1:} $0 < w_{k,p} \leq \upsilon_p$, for some sequence $(\upsilon_p)_p$, where $\upsilon_p \rightarrow 0$ for $p \rightarrow + \infty$.
\item \emph{AS2:} The diameters according to the $||.||_2$ norm on $\R^d$ of the strata are such that $\max_k Diam(\Omega_{k,p}) \leq D(w_{k,p})$, for some real valued function $D(\cdot)$, such that $D(w) \rightarrow 0$ for $w \rightarrow 0$.
\end{itemize}
If the functions $m$ and $s$ are in $\mathbb{L}_2([0,1]^d)$, then
\begin{equation*}
\lim_{p \rightarrow +\infty}\Sigma_{\N_p}= \inf_{\N measurable} \Sigma_{\N} = \int_{[0,1]^d} s(x)dx,
\end{equation*}
which implies that $n\times Q_{n,\N_p} \rightarrow 0$ for $p \rightarrow +\infty$.
\end{proposition}
\begin{proof}[Sketch of Proof. The full proof is in the Supplementary material (Appendix~\ref{s:cvdimd})]
The form of the model and the definition of $\si_k$ imply that
\begin{align}\label{eq:exprsigma}
\si_k^2 &= \frac{1}{w_k} \int_{\Omega_k} \big(f(x) - \frac{1}{w_k} \int_{\Omega_k} f(u)du \big)^2 dx + \frac{1}{w_k} \int_{\Omega_k} s(x)^2dx.
\end{align}
We first prove that the result hold for uniformly continuous functions, and then generalize to $\mathbb{L}_2$ functions based on a density argument.

{\em Step 1: Convergence when $m$ and $s$ are uniformly continuous:}
Assume that $m$ and $s$ are uniformly continuous with respect to the $||.||_2$ norm. For any $\upsilon >0$, there exists $\eta \hspace{0.2cm} s.t. \hspace{0.2cm} \forall x, |s(x+u) - s(x)|\leq \upsilon$ and $|f(x+u) -f(x)|\leq \upsilon$ where $u \in \B_{2,d}(\eta)$.
We choose $K$ large enough so that the size of the strata is smaller than $\upsilon$, and their diameter is smaller than $\eta$ (it is possible to do so since the diameter of the strata shrinks to $0$ as $K\rightarrow\infty$).
From Equation~\ref{eq:exprsigma} we deduce that
\begin{align*}
\si_{k}^2 - (\frac{1}{w_{k}} \int_{\Omega_{k}} s)^2 &\leq  2\upsilon^2,
\end{align*}
and using the concavity of the square-root function, we have $\sum_k w_{k}\si_{k} - \int_{[0,1]^d} s  \leq \sqrt{2} \upsilon$, which concludes the proof for uniformly continuous functions.

{\em Step 2: Generalization to the case where $f$ and $s$ are in $\mathbb{L}_2([0,1)^d)$:}
From the density property of the uniformly continuous functions in $\mathbb{L}_{2}([0,1]^d)$ (with respect to the $||.||_2$ norm), we deduce that for any $K$ and $\upsilon$, there exists two uniformly continuous function $f_{\upsilon}$ and $s_{\upsilon}$ such that:
\begin{equation*}
\Bigg|\sum_{k=1}^{K} w_{k}\si_{k} - \sum_{k=1}^{K} \sqrt{w_{k}}\sqrt{ \int_{\Omega_{k}} \Big(f_{\upsilon}(x) + \int_{\Omega_{k}} f_{\upsilon}(u)du\Big)^2dx   - \frac{1}{w_{k}} \int_{\Omega_{k}} s_{\upsilon}^2(x) dx}\Bigg| \leq \upsilon,
\end{equation*}
and also that $\int_{\Omega} |s(x)-s_{\upsilon}(x)| dx   \leq \sqrt{\frac{\upsilon}{2}}.$ One concludes by combining those two inequalities with Step 1.
\end{proof}

In Proposition~\ref{lem:absconv}, even though the optimal oracle allocation might not be realizable (in particular if the number of strata is larger than the budget), we can still compute the quality of a partition, as defined in~\ref{eq:quality}. It does not correspond to any reachable pseudo-risk, but rather to a lower bound on any (even oracle) static allocation.

When $f$ and $s$ are in $\mathbb L_2([0,1]^d)$, for any appropriate sequence of partitions $(\N_p)_p$, $\Sigma_{\N_p}$ (which is the principal ingredient of the variance of the optimal oracle allocation) converges to the smallest possible $\Sigma_{\N}$ for given $f$ and $s$. Note however that this condition is not sufficient to obtain a \textit{rate}.

\paragraph{Finite-Time analysis under H\"older assumption:}

We make the following assumption on the functions $f$ and $s$.

\begin{assumption}\label{ass:hold}
 The functions $f$ and $s$ are $(M, \alpha)-$H\"older continuous, i.e., for $g\in\{m,s\},$ for any $x$ and $y\in [0,1]^d, |g(x) - g(y)| \leq M ||x-y||_2^{\alpha}$.
\end{assumption}

The H\"older assumption enables to consider arbitrarily non-smooth functions (for small $\alpha$, the function can vary arbitrarily fast), and is thus a fairly general assumption.

We also consider the following partitions in $K$ squared strata.
\begin{definition}\label{ass:stratas}
We write $\N_K$ the  partition of $[0,1]^d$ in $K$ hyper-cubic strata of measure $w_k = w = \frac{1}{K}$ and side length $(\frac{1}{K})^{1/d}$: we assume for simplicity that there exists an integer $l$ such that $K = l^d$. 
\end{definition}

The following Proposition holds.

\begin{proposition}\label{prop:noisy}
Under Assumption~\ref{ass:hold} we have for any partition $\N_K$ as defined in Definition~\ref{ass:stratas} that
\begin{equation}\label{eq:prop3.1}
\Sigma_{\N_K} - \int_{[0,1]^d} s(x)dx \leq \sqrt{2 d} M (\frac{1}{K})^{\alpha/d},
\end{equation}
which implies
\begin{equation*}
Q_{n, \N_K} \leq \frac{2\sqrt{2d}M\Sigma_{\N_1}}{n}(\frac{1}{K})^{\alpha/d},
\end{equation*}
where $\N_1$ stands for the ``partition'' with one stratum.
\end{proposition}

\begin{proof}[Sketch of Proof for Proposition~\ref{prop:noisy}]
We deduce from Assumption~\ref{ass:hold} that
\begin{align*}
\frac{1}{w_{k}}\int_{\Omega_{k}} \Big(f(x) - \frac{1}{w_{k}} \int_{\Omega_{k}} f(u)du \Big)^2dx +\frac{1}{w_{k}} \int_{\Omega_{k}} s^2(x)dx - \big(\frac{1}{w_{k}} \int_{\Omega_{k}} s(u)du\big)^2 &\leq 2M^2d (\frac{1}{K})^{2\alpha/d}.
\end{align*}
Then, by using Equation~\ref{eq:exprsigma} and by summing over all strata, we deduce Equation~\ref{eq:prop3.1}. Now the result on the quality follows from the fact that $\Sigma_{\N_K}^2 - \big(\int_{[0,1]^d} s(x)dx\big)^2 = (\Sigma_{\N_K}^2 - \big(\int_{[0,1]^d} s(x)dx\big)^2 ) (\Sigma_{\N_K}^2 + \big(\int_{[0,1]^d} s(x)dx\big)^2 ) \leq 2 \Sigma_{\N_1}(\Sigma_{\N_K}-\int_{[0,1]^d} s(x)dx)$.
\end{proof}
The full proof for this Proposition is in the Supplementary Material (Appendix~\ref{ss:ratelip}).

\subsection{General comments}\label{sec:general.comments}

\paragraph{The impact of $\alpha$ and $d$:} The quantity $Q_{n,\N_K}$ increases with the dimension $d$, because the H\"older assumption becomes less constraining when $d$ increases. This can easily be seen since a squared strata of measure $w$ has a diameter of order $w^{1/d}$. $Q_{n,\N_K}$ decreases with the smoothness $\alpha$ of the function, which is a logic effect of the H\"older assumption. Note also that when defining the partitions $\N_K$ in Definition~\ref{ass:stratas}, we made the crucial assumption that $K^{1/d}$ is an integer. This fact is of little importance in small dimension, but will matter in high dimension, as we will enlighten in the last remark of Section~\ref{s:tradeoff}.

\paragraph{Minimax optimality of this rate:} The rate $n^{-1} K^{-\alpha/d}$ is minimax optimal on the class of $\alpha-$H\"older functions since for any $n$ and $K$ one can easily build a function with H\"older exponent $\alpha$ such that the corresponding $\Sigma_{\N_K}$ is at least $\int_{[0,1]^d} s(x)dx + c K^{-\alpha/d}$ for some constant $c$.

\paragraph{Discussion on the shape of the strata:} Whatever the shape of the strata, as long as their diameter goes to $0$\footnote{And note that in this \textit{noisy} setting, if the diameter of the strata does not go to $0$ on non homogeneous part of $m$ and $s$, then the standard deviation corresponding to the allocation is larger than $\int_{[0,1]^d} s(u)du$.}, $\Sigma_{\N_K}$ converges to $\int_{[0,1]^d} s(x)dx$. The shape of the strata have an influence only on the negligible term, i.e.~the speed of convergence to this quantity. This result was already made explicit, in a different setting and under different assumptions, in \citep{A-EtoForJouMou11}. Choosing small strata of same shape and size is also minimax optimal on the class of H\"older functions. Working on the shape of the strata could, however, improve the speed of convergence in some specific cases, e.g.~when the noise is very localized. It could also be interesting to consider strata of varying size, and make this size depend on the specific problem. %This could be the object of future works.

\paragraph{The decomposition of the variance:} Note that the variance $\sigma_k^2$ within each stratum $\Omega_k$ comes from two sources. First, $\sigma_k^2$ comes from the noise, that contributes to  it by $\frac{1}{w_k} \int_{\Omega_k} s(x)^2 dx$. Second, the mean $f$ is not a constant function, thus its contribution to $\sigma_k^2$  is $\frac{1}{w_k} \int_{\Omega_k} \big(f(x) - \frac{1}{w_k} \int_{\Omega_k} f(u)du\big)^2 dx$. Note that when the size of $\Omega_k$ goes to $0$, this later contribution vanishes, and the optimal allocation is thus proportional to $\sqrt{w_k \int_{\Omega_k} s(x)^2 dx + o(1)} = \int_{\Omega_k} s(x)dx + o(1)$. This means that for small strata, the variation in the mean are negligible when compared to the variation due to the noise.

\section{Algorithm MC-UCB and a matching lower bound}\label{s:algo}

\subsection{Algorithm $MC-UCB$}\label{ss:algo}

In this Subsection, we describe a slight modification of the algorithm $MC-UCB$ introduced in \citep{MC-UCB}. The only difference is that we change the form of the high-probability upper confidence bound on the standard deviations, in order to improve the elegance of the proofs, and we refine their analysis. The algorithm takes as input two parameters $b$ and $f_{\max}$ which are linked to the distribution of the arms, $\delta$ which is a (small) probability, and the partition $\N_K$. %The upper confidence bound is there obtained by using an analysis of the exponential moments. 
We remind in Figure~\ref{f:m-algorithm} the algorithm $MC-UCB$.

\begin{figure}[ht]
\bookbox{
\begin{algorithmic}
\STATE \textbf{Input:} $b$, $f_{\max}$, $\delta$, $\N_K$, set $A = 2\sqrt{(1 + 3b + 4f_{\max}^2)\log(2nK/\delta)}$
\STATE \textbf{Initialize:} Sample $2$ states in each strata.
\FOR{$t = 2K+1,\ldots,n$}
  \STATE Compute $B_{k,t} = \frac{w_k}{T_{k,t-1}} \Big( \hsi_{k,t-1} + A \sqrt{\frac{1}{T_{k,t-1}}}  \Big)$ for each stratum $k \leq K$
  \STATE Sample a point in stratum $k_t\in\arg\max_{1\leq k \leq K} B_{k,t}\quad$
\ENDFOR
\STATE \textbf{Output:} $\hmu_{n} = \sum_{k=1}^K w_k \hmu_{k,n}$
\end{algorithmic}}
\caption{The pseudo-code of the MC-UCB algorithm. The empirical standard deviations and means $\hsi_{k,t}^2$ and $\hmu_{k,t}$ are computed using Equations~\ref{eq:estim-var2} and ~\ref{eq:estim-mean2}.}\label{f:m-algorithm}
\end{figure}

The estimates of $\hsi_{k,t-1}^2$ and $\hmu_{k,t-1}$ are computed according to
\begin{equation}\label{eq:estim-var2}
\hsi_{k,t-1}^2 = \frac{1}{T_{k,t-1}} \sum_{i=1}^{T_{k,t-1}} (X_{k,i}-\hmu_{k,t-1})^2\;,
\end{equation}
and
\begin{equation}\label{eq:estim-mean2}
\hmu_{k,t-1} = \frac{1}{T_{k,t-1}} \sum_{i=1}^{T_{k,t-1}} X_{k,i}\;.
\end{equation}

% We state the following assumption on the sample $X_{k,t}$
% 
% 
% 
% \begin{assumption}\label{ass:distr}
% $\exists b$ such that $\forall k \leq K$, $\forall t \leq n$, and $\forall \lambda <\frac{1}{b}$,
% 
% 
% $$\E\Big[ \exp(\lambda (X_{k,t}-\mu_k) ) \Big] \leq \exp\Big( \frac{\lambda^2 \si_k^2}{2(1 - \lambda b)}\Big),$$
% 
% $$\E\Big[ \exp(\lambda (X_{k,t}-\mu_k)^2 - \lambda \si_k^2) \Big] \leq \exp\Big( \frac{\lambda^2 \si_k^2}{2(1 - \lambda b)}\Big).$$ 
% 
% \end{assumption}
% 

%We could state Assumption~\ref{ass:distr} in terms of the noise $W_x$, and also add some conditions on the functions $m$ and $s$. Assumption~\ref{ass:distr} would then be implied by

\subsection{Upper bound on the pseudo-regret of algorithm MC-UCB.}

We first state the following Assumption on the noise $\no_t$:
\begin{assumption}\label{ass:distr2}
There exist $b>0$ such that $\forall x \in [0,1]^d$, $\forall t$, and $\forall \lambda <\frac{1}{b}$,
$$\E_{\nu_x}\Big[ \exp(\lambda \no_t ) \Big] \leq \exp\Big( \frac{\lambda^2 }{2(1 - \lambda b)}\Big),
\quad \mbox{ and }\quad 
\E_{\nu_x}\Big[ \exp(\lambda \no_t^2 - \lambda) \Big] \leq \exp\Big( \frac{\lambda^2 }{2(1 - \lambda b)}\Big).$$
\end{assumption}

This is a kind of sub-Gaussian assumption, satisfied for e.g., Gaussian as well as bounded distributions.
We also state an assumption on $f$ and $s$.

\begin{assumption}\label{ass:bound}
The functions $f$ and $s$ are bounded by $f_{\max}$.
\end{assumption}

Note that since the functions $f$ and $s$ are defined on $[0,1]^d$, if Assumption~\ref{ass:hold} is satisfied, then Assumption~\ref{ass:bound} holds with $f_{\max} = \max(f(0),s(0)) + M d^{\alpha/2}$.
We now prove the following bound on the pseudo-regret. Note that we state it on partitions $\N_K$, but that it in fact holds for any partition in $K$ strata.

\begin{proposition}\label{prop:m-regret}

Under Assumptions~\ref{ass:distr2} and~\ref{ass:bound}, on partition $\N_K$, when $n\geq 4K$, we have %with probability $1-\delta$,
\begin{equation*}
\E[R_{n,\N_K}(\alg_{MC-UCB})] \leq  %\frac{4 \Sigma_{\N_K}\sqrt{2}A}{\sqrt{B}} \frac{K^{1/3}}{n^{4/3}} + \frac{12 K \Sigma_{\N_K}^2}{n^2},
 24\sqrt{2} \Sigma_{\N_K}\sqrt{(1 + 3b + 4 f_{\max}^2)}\Big(\frac{f_{\max} +  4}{4}\Big)^{1/3} \frac{K^{1/3}}{n^{4/3}}\sqrt{\log(nK)} + \frac{14K\Sigma_{\N_K}^2}{n^2}.
\end{equation*}
%where $A$ and $B$ depend of $\bar V$, $b$ and polynomially on $\log(nK/\delta)$.
\end{proposition}

The proof, given in the Supplementary Material (Appendix~\ref{s:m-results}), is close to the one of MC-UCB in~\citep{MC-UCB}. But an improved analysis leads to a better dependency in terms of number of strata $K$. We remind that in paper~\citep{MC-UCB}, the bound is of order $\tilde O(K n^{-4/3})$.
This improvement is crucial here since the larger $K$ is, the closer $\Sigma_{\N_K}$ is from $\int_{[0,1]^d} s(x) dx$. The next Subsection states that the rate $K^{1/3}\tilde O(n^{-4/3})$ of MC-UCB is optimal both in terms of $K$ and $n$.

%The fact that the regret is of order $O(\frac{K_{\N}^{1/3}}{n^{4/3}})$ is interesting because of as long as $K_{\N} \leq n$ (which is reasonable), the regret is negligible when compared to $\frac{1}{n}$.

\subsection{Lower Bound}

We now study the minimax rate for the pseudo-regret of any algorithm on a given partition $\N_K$. Note that we state it for partitions $\N_K$, but that it holds for any partition in $K$ strata of equal measure.

\begin{theorem}\label{th:lowbound}
Let $K \in \mathbb N$. Let $\inf$ be the infimum taken over all online stratified sampling algorithms on $\N_K$ and $\sup$ represent the supremum taken over all environments, then:
\begin{equation*}
 \inf \sup \E[ R_{n,\N_K}] \geq C\frac{K^{1/3}}{n^{4/3}},
\end{equation*}
where $C$ is a numerical constant.

\end{theorem}
\begin{proof}[Sketch of proof (The full proof is reported in Appendix~\ref{app:lowbound})]
We consider a partition with $2K$ strata. On the $K$ first strata, the samples are drawn from Bernoulli distributions of parameter $\mu_k$ where $\mu_k \in \{\frac{\mu}{2}, \mu,3\frac{\mu}{2}\}$, and on the $K$ last strata, the samples are drawn from a Bernoulli of parameter $1/2$. We write $\sigma = \sqrt{\mu(1-\mu)}$ the standard deviation of a Bernoulli of parameter $\mu$. We index by $\upsilon$ a set of $2^K$ possible environments, where $\upsilon = (\upsilon_1, \ldots, \upsilon_K) \in \{-1,+1\}^K$, and the $K$ first strata are defined by $\mu_k = \mu + \upsilon_k \frac{\mu}{2}$. Write $\P_{\sigma}$ the probability under such an environment, also consider $\P_{\sigma}$ the probability under which all the $K$ first strata are Bernoulli with mean $\mu$.

We define $\Omega_{\upsilon}$ the event on which there are less than $\frac{K}{3}$ arms not pulled correctly for environment $\upsilon$ (i.e.~for which $T_{k,n}$ is larger than the optimal allocation corresponding to $\mu$ when actually $\mu_k=\frac{\mu}{2}$, or smaller than the optimal allocation corresponding to $\mu$ when $\mu_k=3\frac{\mu}{2}$). See the Appendix~\ref{app:lowbound} for a precise definition of these events.
Then, the idea is that there are so many such environments that any algorithm will be such that for at least one of them we have $\P_{\sigma}(\Omega_{\upsilon}) \leq \exp(-K/72)$. Then we derive by a variant of Pinsker's inequality applied to an event of small probability that $\P_{\upsilon}(\Omega_{\upsilon}) \leq \frac{KL(\P_{\sigma}, \P_{\upsilon})}{K} = O(\frac{\sigma^{3/2} n}{K})$. Finally, by choosing $\sigma$ of order $(\frac{K}{n})^{1/3}$, we have that $\P_{\upsilon}(\Omega_{\upsilon}^c)$ is bigger than a constant, and on $\Omega_{\upsilon}^c$ we know that there are more than $\frac{K}{3}$ arms not pulled correctly. This leads to an expected pseudo-regret in environment $\upsilon$ of order $\Omega(\frac{K^{1/3}}{n^{4/3}})$.
\end{proof}

This is the first lower-bound for the problem of online stratified sampling for Monte-Carlo. Note that this bound is of same order as the upper bound for the pseudo-regret of algorithm MC-UCB. It means that this algorithm is, up to a constant, minimax optimal, both in terms of the number of samples and in terms of the number of strata. It however holds only on the partitions $\N_K$ (we conjecture that a similar result holds for \textit{any} measurable partition $\N$, but with a bound of order $\Omega\Big(\sum_{x \in \N} \frac{w_x^{2/3}}{n^{4/3}}\Big)$).

%However it is important to note that the upper bound for algorithm $MC-UCB$ is of order $\sum_{k=1}^K \frac{w_k^{2/3}}{n^{1/3}}$, which is equal to $\frac{K^{1/3}}{n^{4/3}}$ only when all the strata have equal size. In general we have $\sum_{k=1}^K w_k^{2/3}\leq K^{1/3}$ but both quantities can be significantly different, such as for instance when $K-1$ strata are of size $\frac{1}{2^k}$ and one of size $\frac{1}{2^{K-1}}$. This leaves open the question of finding a matching ``strata-dependent'' lower bound.

\section{Best trade-off between $Q_{n,\N_K}$ and $R_{n,\N_K} (\alg_{MC-UCB})$}\label{s:tradeoff}

\subsection{Best trade-off}

We consider in this Section the hyper-cubic partitions $\N_K$ as defined in Definition~\ref{ass:stratas}, and we want to find the best number of strata $K_n$ as a function of $n$.
Using the results in Section~\ref{s:conver}~and Subsection~\ref{ss:algo}, it is possible to deduce an optimal number of strata $K$ to give as parameter to algorithm $MC-UCB$. Note that since the performance of the algorithm is defined as the sum of the quality of partition $\N_K$, i.e.~$Q_{n,\N_K}$ and of the pseudo-regret of the algorithm MC-UCB, namely $R_{n,\N_{K}}(\alg_{MC-UCB})$, one wants to (i) on the one hand take many strata so that $Q_{n,\N_{K}}$ is small but (ii) on the other hand, pay attention to the impact this number of strata has on the pseudo-regret $R_{n,\N_{K}}(\alg_{MC-UCB})$. A good way to do that is to choose $K_n$ in function of $n$ such that $Q_{n,\N_{K_n}}$ and $R_{n,\N_{K_n}}(\alg_{MC-UCB})$ are of the same order.

\begin{theorem}\label{th:tradeoff} 
Under Assumptions~\ref{ass:hold} and~\ref{ass:distr2} (since on $[0,1]^d$, Assumption~\ref{ass:hold} implies Assumption~\ref{ass:bound}, by setting $f_{\max} = X(1) + Md^{\alpha/2}$), choosing $K_n = \Big(\lfloor (n^{\frac{d}{d+3\alpha}})^{1/d} \rfloor \Big)^d (\leq n^{\frac{d}{d+3\alpha}} \leq n)$, we have %w.p.~$1-\delta$
\begin{align*}
 \E[ L_n(\alg_{MC-UCB})] - \frac 1n \Big(\int_{[0,1]^d} s(x) dx\Big)^2 %= \tilde O(\sqrt{d}n^{-\frac{d+4\alpha}{d+3\alpha}} (1+d^{\alpha/d}n^{-\frac{\alpha}{d + 3 \alpha}}))
\leq  C  d^{\frac{2\alpha}{3d} + \frac12} \sqrt{\log(n)} n^{-\frac{d+4\alpha}{d+3\alpha}} (1+d^{\alpha}n^{-\frac{\alpha}{d + 3 \alpha}}),
\end{align*}
where $c = 70 (1+M) \Sigma_{\N_K}\sqrt{(1 + 3b + 4 (f(0) + s(0) + M)^2)}\Big(\frac{(f(0) + s(0) + M) +  4}{4}\Big)^{1/3}$.
\\
This leads to, if $d \ll n$, the simplified bound is
\begin{equation*}
 \E[ L_n(\alg_{MC-UCB})] - \frac 1n \Big(\int_{[0,1]^d} s(x) dx\Big)^2 = \tilde O(n^{-\frac{d+4\alpha}{d+3\alpha}}).
\end{equation*}
%Here the $\tilde O$ depends of $f_{\max}$, $b$, $M$, $\alpha$ and polynomially on $\log(n)$.
\end{theorem}
\begin{proof}[Proof of Theorem~\ref{th:tradeoff}]
The definition of $K_n$ implies that $K_n \geq \Big( (n^{\frac{d}{d+3\alpha}} -1 )^{1/d} \Big)^d \geq n^{\frac{d}{d+3\alpha}} \Big(1 - \frac{d}{n^{\frac{1}{d}(\frac{d}{d+3\alpha})}} \Big)$. Also, trivially, $K_n \leq n^{\frac{d}{d+3\alpha}}$. By plugging these lower and upper bounds, in respectively $Q_{n, \N_{K_n}}$ and $R_{n, \N_{K_n}}$, we obtain the the final bound.
%Let us consider the set of $\alpha-$H\"older functions, from which one observes noisy samples from a Bernoulli of parameter $m(x)$. For every function that attains the H\"older coefficient $\alpha$ in any point, whatever the stratification $\N$ of the space in $K$ convex strata, $Q_{n,\N} = O(\frac{L}{nK^{\alpha/d}})$. Combining that with the results of Theorem~\ref{th:lowbound}, one obtains the lower bound.
\end{proof}
We can also prove a matching minimax lower bound using the results in Theorem~\ref{th:lowbound}.
\begin{theorem}\label{th:lowbound2}
 Let $\sup$ represent the supremum taken over all $\alpha-$H\"older functions and $\inf$ be the infimum taken over all algorithms that partition the space in convex strata of same shape, then the following holds true:
\begin{equation*}
\inf \sup \E L_n(\alg) - \frac 1n \Big(\int_{[0,1]^d} s(x) dx\Big)^2 = \Omega(n^{-\frac{d+4\alpha}{d+3\alpha}}).
\end{equation*}
\end{theorem}
\begin{proof}[Proof of Theorem~\ref{th:lowbound2}]
This is a direct consequence of Theorem~\ref{th:lowbound} and the second comment of Subsection~\ref{sec:general.comments}.
%Let us consider the set of $\alpha-$H\"older functions, from which one observes noisy samples from a Bernoulli of parameter $m(x)$. For every function that attains the H\"older coefficient $\alpha$ in any point, whatever the stratification $\N$ of the space in $K$ convex strata, $Q_{n,\N} = O(\frac{L}{nK^{\alpha/d}})$. Combining that with the results of Theorem~\ref{th:lowbound}, one obtains the lower bound.
\end{proof}
%Note that when $d$ is very high, $\Sigma_{\N_{K_n}}$ converges more slowly to $\int_{[0,1]^d} s(x) dx$ and the number of groups has thus to be chosen higher: yet another illustration of the curse of dimensionality. Note also that there is also an influence of the H\"older exponent $\alpha$, i.e.~the smoother the function, the faster the convergence.

\subsection{Discussion}

\paragraph{Optimal pseudo-risk.} The dominant term in the pseudo-risk of MC-UCB with proper number of strata is $\frac{(\inf_{\N} \Sigma_{\N})^2}{n} = \frac 1n \big(\int_{[0,1]^d} s(x) dx\big)^2$ (the other term is negligible). This means that algorithm MC-UCB is almost as efficient as the optimal oracle strategy on the optimal oracle partition. In comparison, the variance of the estimate given by crude Monte-Carlo is $\int_{[0,1]^d} \big(f(x) -\int_{[0,1]^d} f(u) du\big)^2 dx + \int_{[0,1]^d} s(x)^2 dx$. Thus MC-UCB enables to have the term coming from the variations in the mean vanish, and the noise term decreases (since by Cauchy-Schwarz, $\big(\int_{[0,1]^d} s(x) dx \big)^2 \leq \int_{[0,1]^d} s(x)^2 dx$).

\paragraph{minimax-optimal trade-off for algorithm MC-UCB.} The optimal trade-off on the number of strata $K_n$ of order $n^{\frac{d}{d+3\alpha}}$ depends on the dimension and the smoothness of the function. The higher the dimension, the more strata are needed in order to have a decent speed of convergence for $\Sigma_{\N_{K}}$. The smoother the function, the less strata are needed.
\\
It is yet important to remark that this trade-off is not exact. %in the sense that we provide a minimax-optimal value of $K_n$ up to a $\sqrt{\log(n)}$. But 
We provide an almost minimax-optimal order of magnitude for $K_n$, in terms of $n$, so that the rate of convergence of the algorithm is minimax-optimal up to a $\sqrt{\log(n)}$.  %We can prove convergence of the risk rescaled by $n$ to the optimal constant $\big(\int_{[0,1]^d} s(x) dx\big)^2$.

%\paragraph{Minimax trade-off.} The lower bound on the pseudo-regret for the problem of stratified Monte-Carlo, stated in Theorem~\ref{th:lowbound}, matches the upper bound of algorithm MC-UCB: it is thus a minimax algorithm for the problem of stratified Monte-Carlo allocation. As the rate on the quality provided in Proposition~\ref{prop:noisy} is also minimax on the class of $\alpha$-H\"older functions, the upper bound in Theorem~\ref{th:tradeoff} matches also the lower-bound on the difference between the pseudo-risk and the best asymptotic variance: algorithm $MC-UCB$ with $K_n = n^{\frac{d}{d+3\alpha}}$ is minimax on the class of H\"older functions.

\paragraph{Link between risk and pseudo-risk.} It is important to compare the pseudo-risk $L_n(\alg) = \sum_{k=1}^K \frac{w_k^2\si_k^2}{T_{k,n}}$ and the true risk $\E[ (\hmu_n - \mu)^2 ]$. Note that those quantities are in general not equal for an algorithm $\alg$ that allocates the samples in a dynamic way: indeed, the quantities $T_{k,n}$ are in that case stopping times and the variance of estimate $\hmu_n$ is not equal to the pseudo-risk. However, in the paper~\citep{rapp-tech-MC-UCB}, the authors highlighted for $MC-UCB$ some links between the risk and the pseudo-risk. More precisely, they established links between $L_n(\alg)$ and $\sum_{k=1}^K w_k^2 \E[(\hmu_{k,n}-\mu_k)^2]$. This step is possible since $\E[(\hmu_{k,n} - \mu_k)^2] \leq \frac{w_k^2 \si_k^2}{\underline{T}_{k,n}^2} \E[T_{k,n}]$, where $\underline{T}_{k,n}$ is a lower-bound on the number of pulls $T_{k,n}$ on a high probability event. Then they bounded the cross products $\E[(\mu_{k,n} - \mu_k)(\hmu_{p,n}- \mu_p)]$ and provided some upper bounds on those terms. A tight analysis of these terms as a function of the number of strata $K$ remains to be investigated.

\paragraph{Knowledge of the H\"older exponent.} In order to be able to choose properly the number of strata to achieve the rate in Theorem~\ref{th:tradeoff}, it is needed to possess a proper lower bound on the H\"older exponent of the function: indeed, the rougher the function is, the more strata are required. On the other hand, such a knowledge on the function is not always available and an interesting question is whether it is possible to estimate this exponent fast enough. There are interesting papers on that subject like~\citep{hoffmann2002random} where the authors tackle the problem of regression and prove that it is possible, up to a certain extent, to adapt to the unknown smoothness of the function. The authors in~\citep{gine2010confidence} add to that (in the case of density estimation) and prove that it is even possible under the assumption that the function attain its H\"older exponent to have a proper estimation of this exponent and thus adaptive confidence bands. An idea would be to try to adapt those results in the case of finite sample.

\paragraph{MC-UCB On a noiseless function.} Consider the case where $s=0$ almost surely, i.e.~the samples collected are noiseless. Proposition~\ref{lem:absconv} ensures that $\inf_{\N} \Sigma_{\N} = 0$: it is thus possible in this case to achieve a pseudo-risk that has a faster rate than $O(\frac{1}{n})$. If the function $m$ is smooth, e.g.~H\"older with a not too low exponent $\alpha$, it is efficient to use low discrepancy methods to integrate the functions. An idea is to stratify the domain in $n$ hyper-rectangular strata of minimal diameter, and to pick at random one sample per stratum. The variance of the resulting estimate is of order $O(\frac{1}{n^{1+2\alpha/d}})$. Algorithm MC-UCB is not as efficient as a low discrepancy schemes: it needs a number of strata $K < n$ in order to be able to estimate the variance of each stratum. Its pseudo-risk is then of order $O(\frac{1}{n K^{2\alpha/d}})$.
\\
It is however only true when the observations are noiseless. Otherwise, the order for the variance of the estimate is in $1/n$, no matter what strategy the learner chooses.

\paragraph{In high dimension.} The first bound in Theorem~\ref{th:tradeoff} expresses precisely how the performance of the estimate outputted by MC-UCB depends on $d$. The first bound states that the quantity  $L_n(\alg) - \frac 1n \Big(\int_{[0,1]^d} s(x) dx\Big)^2$ is negligible when compared to $1/n$ when $n$ is exponential in $d$. This is not surprising since our technique aims at stratifying equally in every direction. It is not possible to stratify in every directions of the domain if the function lies in a very high dimensional domain.
\\
This is however \textit{not} a reason for not using our algorithm in high dimension. Indeed, stratifying even in a small number of strata already reduces the variance, and in high dimension, any variance reduction techniques are welcome. As mentioned in the end of Section~\ref{s:setting}, the model that we propose for the function is suitable for modeling $d^*$ dimensional functions that we only stratify in $d<d^*$ directions (and $d \ll n$). A reasonable trade-off for $d$ can also be inferred from the bound, but we believe that what a good choice of $d$ is depends a lot of the problem. We then believe that it is a good idea to select the number of strata in the minimax way that we propose. Again, having a very high dimensional function that one stratifies in only a few directions is a very common technique in financial mathematics, for pricing options (practitioners stratify an infinite dimensional process in only 1 to 5 carefully chosen dimensions).

\section{Numerical experiment: influence of the number of strata in the Pricing of an Asian option}\label{s:experiments}

%%%%%%%%%%%%%%%%%%%%%%%%%%%%%%%%%%%%%%%%%%%%%%%%%%%%%%%%%%%%%%%%%%%%%%%%%%%%%%%
%% Comparison
%%%%%%%%%%%%%%%%%%%%%%%%%%%%%%%%%%%%%%%%%%%%%%%%%%%%%%%%%%%%%%%%%%%%%%%%%%%%%%%

We consider the pricing problem of an Asian option introduced in \citep{glasserman1999asymptotically} and later considered in \citep{kawai2010asymptotically,A-EtoJou10}.
This uses a Black-Scholes model with strike $C$ and maturity $T$. Let $(W(t))_{0\leq t\leq T}$ be a Brownian motion. The discounted payoff of the Asian option is defined as a function of $W$, by:
\begin{equation}\label{eq:price}
\textstyle{ F((W)_{0 \leq t \leq T}) = \exp(-rT) \max \Big[ \int_0^T S_0 \exp \Big( (r - \frac{1}{2}s_0^2)t + s_0 W_t\Big) dt -C , 0 \Big],
}
\end{equation}
where $S_0$, $r$, and $s_0$ are constants, and the price is defined by the expectation $p = \E_{W} F(W)$.

We want to estimate the price $p$ by Monte-Carlo simulations (by sampling on $W$). In order to reduce the variance of the estimated price, we can stratify the space of $W$. \cite{glasserman1999asymptotically} suggest to stratify according to a one dimensional projection of $W$, i.e., by choosing a time $t$ and stratifying according to the quantiles of $W_t$ (and simulating the rest of the Brownian according to a Brownian Bridge, see~\citep{kawai2010asymptotically}). They further argue that the best direction for stratification is to choose $t=T$, i.e., to stratify according to the last time of $T$. %Thus we sample $W_d$ and then conditionally sample $W_1,...,W_{d-1}$ according to a Brownian Bridge as explained in \citep{kawai2010asymptotically}.
This choice of stratification is also intuitive since $W_T$ has the highest variance, the biggest exponent in the payoff (\ref{eq:price}), and thus the highest volatility. \cite{kawai2010asymptotically} and \cite{A-EtoJou10} also use the same direction of stratification. We stratify according to the quantiles of $W_T$, that is to say the quantiles of a normal distribution $\mathcal N(0,T)$. When stratifying in $K$ strata, we stratify according to the $1/K$-th quantiles (so that the strata are hyper-cubes of same measure).

%Like in \citep{kawai2010asymptotically} we consider $5$ strata of equal weight. Since $W_d$ follows a ${\cal N}(0,1)$, the strata correspond to the $20$-percentile of a normal distribution.
%The left plot of Figure \ref{f:stratas} represents the cumulative distribution function of $W_d$ and shows the strata in terms of percentiles of $W_d$. The right plot represents, in dot line, the curve $\E[F(W) |W_d = x]$ versus $\Prob(W_d < x)$ parameterized by $x$, and the box plot represents the expectation and standard deviations of $F(W)$ conditioned on each stratum. We observe that this stratification produces an important heterogeneity of the standard deviations per stratum, which indicates that a stratified sampling would be profitable compared to a crude Monte-Carlo sampling.

%\begin{figure}[!hbtp]
%\begin{minipage}{0.5\textwidth}
%\includegraphics[width=7cm]{MC-UCBFigs/normal_2.pdf}
%\end{minipage}
%\begin{minipage}{0.5\textwidth}
%\includegraphics[width=7cm]{MC-UCBFigs/classC90.pdf}
%\end{minipage}
%\includegraphics[width=7cm,height=50mm]{rademacherGAFS.eps}
%\caption{Performance of \textit{(left)} B-AS and \textit{(right)} GAFS-MAX for Gaussian and Rademacher distributions. When $\lambda_{\min}$ decreases, the rescaled regret $n^{3/2}R_n$ in problem 1 (two Gaussian arms) remains approximately constant whereas that for problem 2 (a Gaussian arm versus a Rademacher arm) deteriorates.}
%\caption{Left: Cdf of $W_d$ and the definition of the strata. Right: expectation and standard deviation of $F(W)$ conditioned on each stratum for a strike $C=90$.}
%\label{f:stratas}
%\end{figure}

We choose the same numerical values as \cite{kawai2010asymptotically}: $S_0=100$, $r=0.05$, $s_0=0.30$, $T=1$ and $d=16$. We discretize also, as in~\cite{kawai2010asymptotically}, the Brownian motion in $16$ equidistant times, so that we are able to simulate it. We choose $C = 120$. %Note that the strike $C$ of the option has a direct impact on the variability of the strata. Indeed, the larger $C$, the more probable $F(W)=0$ for strata with small $W_d$, and thus, the smaller $\lambda_{\min}$.

%
% \begin{figure*}[!hbtp]
% \begin{minipage}{0.3\textwidth}
% \includegraphics[width=\textwidth]{K60.eps}
% \end{minipage}
% \begin{minipage}{0.3\textwidth}
% \includegraphics[width=\textwidth]{K90.eps}
% \end{minipage}
% \begin{minipage}{0.3\textwidth}
% \includegraphics[width=\textwidth]{K120.eps}
% \end{minipage}
% \caption{Representation of $\E(F((W_t)_t)/W_d=x)$ in function of $\Prob(W_d \leq x)$ for different values of the strike $C$.} \label{fig:diffvalC}
% \end{figure*}

In this paper, we only do experiments for MC-UCB, and exhibit the influence of the number of strata. For a comparison between MC-UCB and other algorithms, see~\citep{MC-UCB}. By studying the range of the $F(W)$, we set the parameter of the algorithm MC-UCB to $A=150 \log(n)$.

For $n=200$ and $n=2000$, we observe the influence of the number of strata in Figure~\ref{fig:diffvalC}. We observe the trade-off that we mentioned between pseudo-regret and quality, in the sense that the mean squared error of the estimate outputted by MC-UCB (when compared to the true integral of $f$) first decreases with $K$ and then increases. Note that, without surprise, for a large $n$ the minimum of mean squared error is reached with more strata. Finally, note that our technique is never outperformed by uniform stratified Monte-Carlo: it is a good idea to try to adapt.

%In the left plot of Figure~\ref{fig:diffvalC}, we plot the rescaled true regret $\bar R_n n^{3/2}$, averaged over $50000$ trials, as a function of $n$, where $n$ ranges from $50$ to $5000$. The value of the strike is $C=120$. Again, we notice that MC-UCB performs better than Uniform and SSAA because it adapts faster to the distributions of the strata. But it performs very similarly to GAFS-WL. In addition, it seems that the true regret of Uniform and SSAA grows faster than the rate $n^{3/2}$, whereas MC-UCB, as well as GAFS-WL, grow with this rate. The right plot focuses on the MC-UCB algorithm and rescales the $y-$axis to observe the variations of its rescaled true regret more accurately. The curve grows first and then stabilizes. This could correspond to the two regimes discussed previously.

%\vspace{-1.5cm}

\begin{figure*}[!hbtp]
\begin{minipage}{0.49\textwidth}
%\vspace{0.4cm}
\includegraphics[height = 10.2cm, width = 7cm]{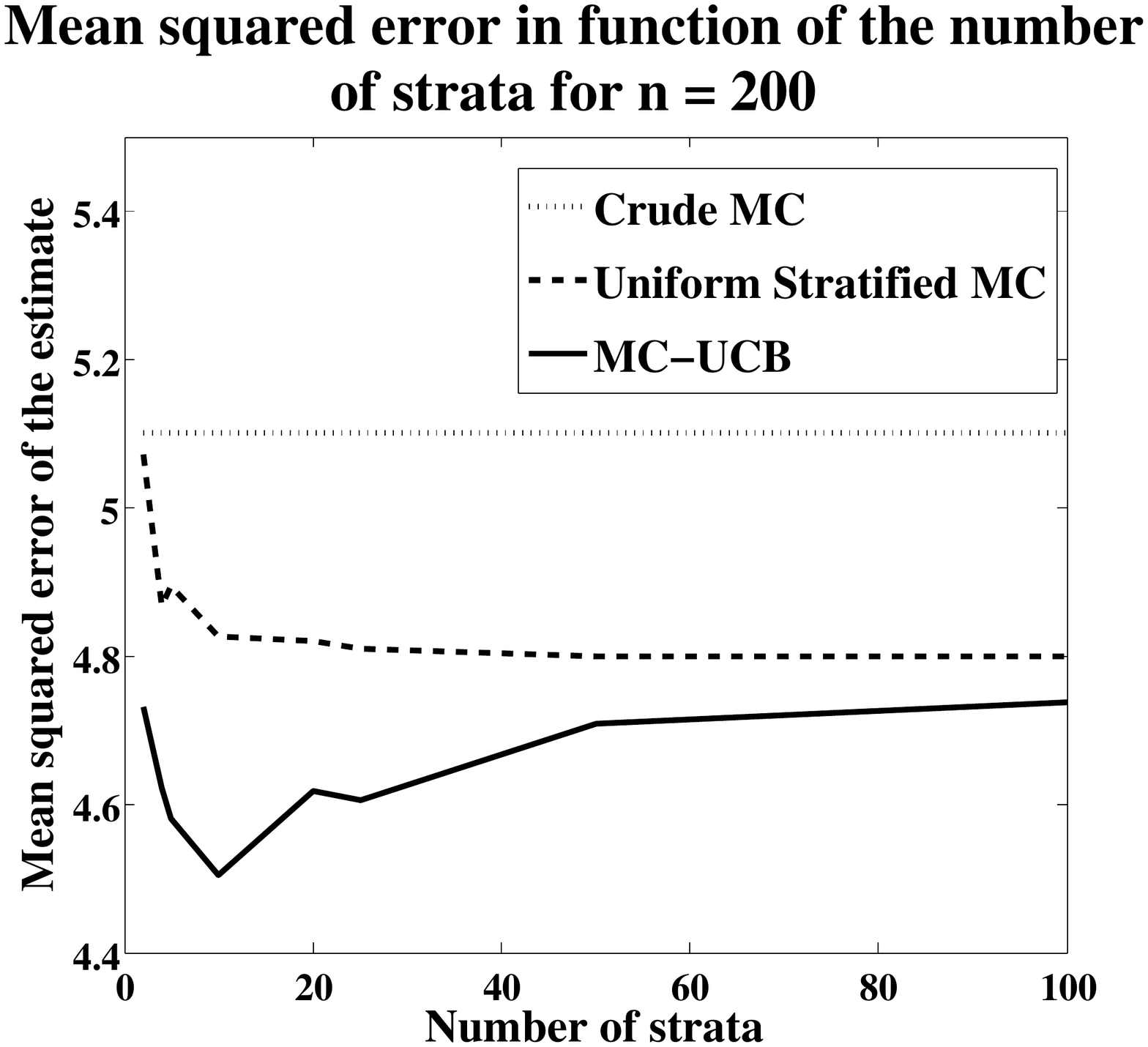}
\end{minipage}
\begin{minipage}{0.48\textwidth}
\includegraphics[width=\textwidth]{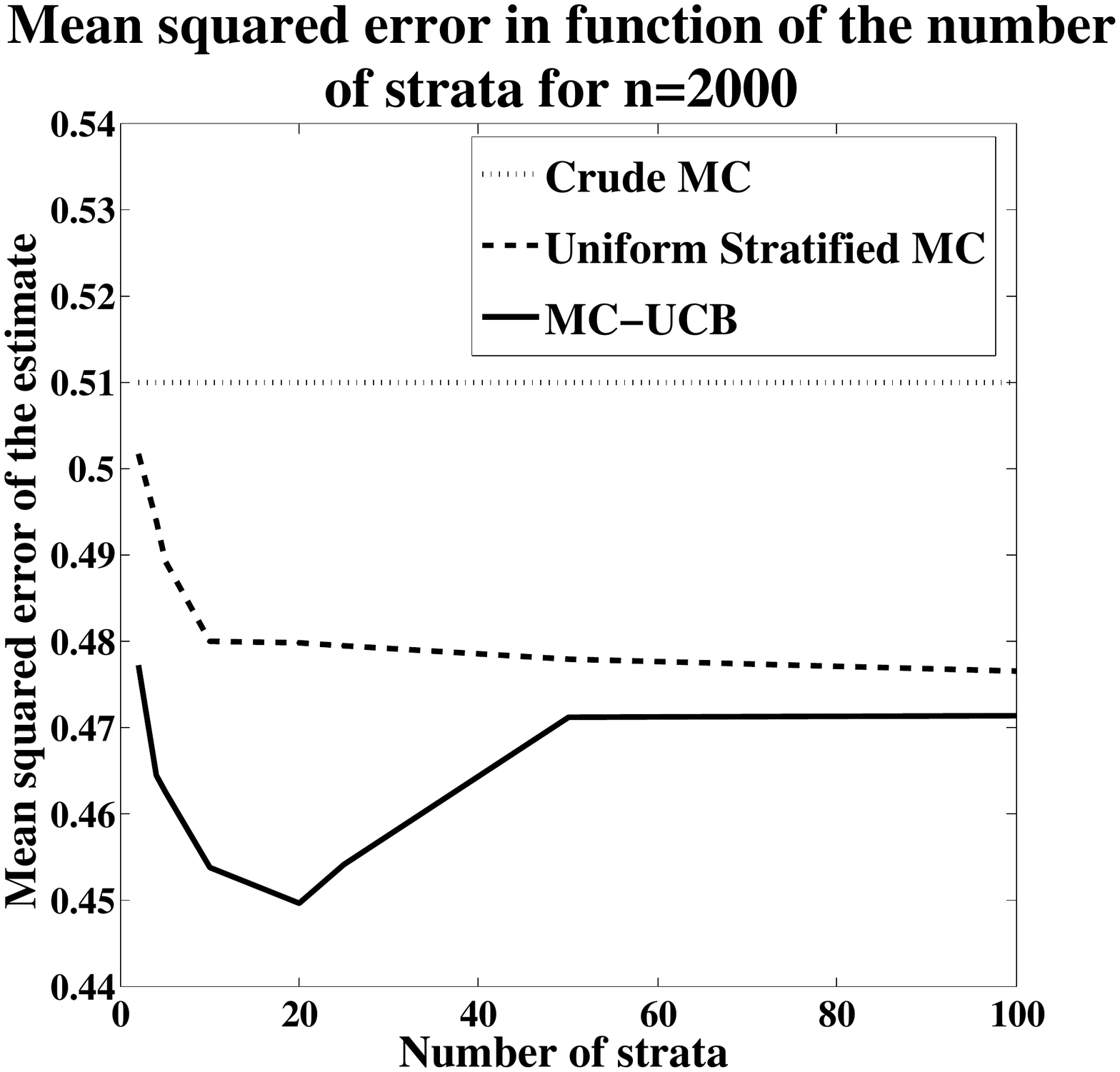}
\end{minipage}
\vspace{-1.5cm}
\caption{Mean squared error for uniform stratified sampling for different number of strata, for (Left:) n=200 and (Right:) n=2000.} \label{fig:diffvalC}
\end{figure*}

\section{Conclusion}

In this paper we studied  the problem of online stratified sampling for the numerical integration of a function given noisy evaluations, and more precisely we discussed the problem of choosing the \textit{minimax-optimal number} of strata.

We explained why, to our minds, this is a crucial problem when one wants to design an efficient algorithm. We enlightened the fact that there is a trade-off between having many strata (and a good approximation error, called the quality of a partition), and not too many, in order to perform almost as well as the optimal oracle allocation on a given partition (small estimation error, called pseudo-regret). %The possibility to perform a finite-time analysis is really important here since the minimax-optimal trade-off is obtained by minimizing the sum of the quality of stratification and the (tight) bound on the pseudo-regret. Note that in the paper of \citep{A-EtoForJouMou11}, since they do an asymptotic analysis, they are not able to provide precise indications on how to calibrate the number of strata so that the best possible trade-off is achievable.
  
When the function is noisy, the noise is the dominant quantity in the optimal oracle variance on the optimal oracle partition. Indeed, decreasing the size of the strata does not diminish the (local) variance of the noise. In this case, the pseudo-risk of algorithm MC-UCB is equal, up to negligible terms, to the mean squared error of the estimate outputted by the optimal oracle strategy on the best (oracle) partition, at a rate of $O(n^{-\frac{d+4\alpha}{d+3\alpha}})$ where $\alpha$ is the H\"older exponent of $s$ and $m$. This rate is minimax optimal on the class of $\alpha$-H\"older functions: it is not possible, up to a constant factor, to do better on simultaneously all $\alpha$-H\"older functions.

%When the function is noiseless, the variance comes only from the variation in $m$, and if $m$ is smooth enough (e.g.~Lipschitz), those variations go to $0$ when the diameter of the strata go to $0$. It is thus possible to show that the pseudo-risk is smaller than $\frac{1}{n}$, and more precisely of order $O(n^{-\frac{d+5}{d+3}})$.

%Note that in the paper of \citep{A-EtoForJouMou11}, they consider a slightly different setting, also using a different presentation. In their paper they consider a deterministic function $f$ in dimension $d$, but they only stratify in $m \leq d$ directions. The case $m <d$ corresponds to the noisy case: indeed, the direction that are not stratified contain some incompressible volatility that is equivalent to a noise. And when $m=d$, it corresponds  to the noiseless case. Note also that as they do an asymptotic analysis, they are not able to provide \textit{precise} indications on how to calibrate the number of strata so that the best possible trade-off is achievable.

We believe that there are (at least) three very interesting remaining open questions:
\begin{itemize}
 \item The first one is to investigate whether it is possible to estimate online the H\"older exponent \textit{fast enough}. Indeed, one needs it in order to compute the proper number of strata for MC-UCB, and the lower bound on the H\"older exponent appears in the bound. It is thus a crucial parameter. %There are some works on the asymptotic computation of this constant. We however need here some finite-time analysis.
\item The second direction is to build a more efficient algorithm in the noiseless case. We remarked that MC-UCB is not as efficient in this case as a simple non-adaptive method. The problem comes from the fact that in the case of a noiseless function, it is important to sample the space in a way that ensures that the points are as spread as possible. An interesting problem is thus to build an algorithm that mixes ideas from quasi Monte-Carlo and ideas from online stratified Monte-Carlo.
\item Another question is the relevance of fixing the strata in advance. Although it is minimax-optimal on the class of $\alpha-$H\"older functions to have hyper-cubic strata of same measure, it might in some cases be more interesting to focus and stratify more finely at places where the function is rough. On that perspective, it could be more clever to have an adaptive procedure that also decides \textit{where} to refine the strata. %But note that MC-UCB adapts the number of samples allocated to each stratum to the variance within the strata, thus there is a strong link between adding more samples in a stratum and refining it. It is thus not straightforward to think of an algorithm with provably better performances.
\end{itemize}

\newpage
\bibliography{allocation.bib}

%%%%%%%%%%%%%%%%%%%%%%%%%%%%%%%%%%%%%%%%%%%%%%%%%%%%%%%%%%%%%%%%%%%%%%%%%%%%%%%
%% APPENDIX
%%%%%%%%%%%%%%%%%%%%%%%%%%%%%%%%%%%%%%%%%%%%%%%%%%%%%%%%%%%%%%%%%%%%%%%%%%%%%%%
\newpage
\appendix

\section{Proof of Theorem~\ref{thm:m-regret}}\label{s:m-results}

\subsection{The main tool: a high probability bound on the standard deviations}

\paragraph{Upper bound on the standard deviation:} 

\begin{lemma}\label{l:event-B-AS}
Let Assumption~\ref{ass:distr2} hold and $n\geq 2$. Define the following event
%
% \begin{equation}\label{eq:ucb-maurer-event-sub1}
% \xi = \bigcap_{1\leq k\leq K,\;1\leq t\leq n}\left\lbrace \Bigg|\sqrt{\frac{1}{2t(t-1)}\sum_{i=1}^t\sum_{j=1}^t(X_{k,i} - X_{k,j})^2} - \si_k\Bigg| \leq 2a\sqrt{\frac{\log(2/\delta)}{t}} \right\rbrace,
% \end{equation}

\begin{equation}\label{eq:ucb-maurer-event-sub1}
\xi = \xi_{K,n}(\delta) = \bigcap_{1\leq k\leq K,\;2\leq t\leq n}\left\lbrace \Bigg|\sqrt{\frac{1}{t-1}\sum_{i=1}^t\Big(X_{k,i} - \frac{1}{t}\sum_{j=1}^t X_{k,j}\Big)^2} - \si_k\Bigg| \leq A\sqrt{\frac{1}{t}} \right\rbrace,
\end{equation}
%
% %
%where $a = 2\sqrt{c_1 \log(c_2/\delta)}$. Then $\Pr(\xi)\geq1-2nK\delta$.
where $A = 2\sqrt{(1 + 3b + 4\bar{V})\log(2 nK/\delta)}$. Then $\Pr(\xi)\geq 1-\delta$.

\end{lemma}
Note that the first term in the absolute value in Equation~\ref{eq:ucb-maurer-event-sub1} is the empirical standard deviation of arm $k$ computed as in Equation~\ref{eq:estim-var2} for $t$ samples. The event $\xi$ plays an important role in the proofs of this section and a number of statements will be proved on this event.

\begin{proof}
Under Assumption~\ref{ass:distr2} we have for $f_{\max}^2 \geq \max_k \si_k^2$ with probability $1-\delta$ because of the results of Lemma~\ref{ss:variance}
\begin{equation}
\Bigg|\sqrt{\frac{1}{t-1}\sum_{i=1}^t\Big(X_{k,i} - \frac{1}{t}\sum_{j=1}^t X_{k,j}\Big)^2} - \si_k\Bigg| \leq  2\sqrt{\frac{(1 + 3b + 4f_{\max}^2)\log(2/\delta)}{t}}. 
\end{equation}
Then by doing a simple union bound on $(k,t)$, we obtain the result.

\end{proof}
We deduce the following corollary when the number of samples $T_{k,t}$ are random.

\begin{corollary}\label{l:bernstein-var-sub}
For any $k=1,\ldots,K$ and $t=2K,\dots,n$, let $\{X_{k,i}\}_i$ be $n$ i.i.d.~random variables drawn from $\nu_k$, satisfying Assumption~\ref{ass:distr2}. Let $T_{k,t}$ be any random variable taking values in $\{2,\dots,n\}$. Let $\hvar_{k,t}$ be the empirical variance computed from Equation~\ref{eq:estim-var2}. Then, on the event $\xi$, we have:
\begin{equation}\label{eq:ucb-maurer-event-sub2}
|\hsi_{k,t} - \si_k| \leq A\sqrt{\frac{1}{T_{k,t}}}\;,
\end{equation}
where $A=2\sqrt{(1 + 3b + 4\bar{V})\log(2nK/\delta)}$.
\end{corollary}

%\paragraph{Great deviation bounds}
%
%Note that the distributions of all our arms are sub-gaussian by assumption. The constants $c_2$ and $c_1$ for each arms might be linked with the $\var_k$, we take $c_2$ and $c_1$ to be the maximum over the $K$ arms.
%
%We now place us in the event $\xi_1$ of probability $1 - nK\delta$ such that $\forall k$, $\forall t \leq n$, $|X_i^k-\mu_k| \leq \sqrt{c_1 \log(c_2/\delta)}$.
%
%Note that on $\xi_1$, $\forall k$, $\forall t \leq n$, $\frac{1}{2\sqrt{c_1 \log(c_2/\delta)}}|X_i^k-\mu_k| \leq 1/2$.
%
%We now transform slightly Theorem 10 from \citep{maurer2009empirical} by just translating the random variables to obtain:
%
%\begin{theorem}
%Let $(X_1,...,X_n)$ be i.i.d. random variables in $[a,a+1]$, of variance $\var$ and mean $\mu$. Let $\hvar_n$ be the empirical variance of the sample.
%Then with probability $1 - \delta$:
%
%$$|\hsi_n - \si| \leq 2\sqrt{\frac{\log(2/\delta)}{n}}.$$
%\end{theorem}
%
%Let us call $\xi_2$ the event  such that $\forall k$, $\forall t \leq n$, we have:
%\begin{equation}\label{eq:bound.bounded.bernstein}
%|\hsi_{k,n} - \si_k| \leq 4\sqrt{c_1 \log(c_2/\delta)}\sqrt{\frac{\log(2/\delta)}{n}}.
%\end{equation}
%
%Note that on $\xi_1$, the probability of event $\xi_2$ is at least $1 - nK\delta$ by the precedent theorem. The probability of $\xi = \xi_1 \bigcap \xi_2$ is thus at least $1 - 2nK\delta$.
%
%Now let us call $a = 4\sqrt{c_1 \log(c_2/\delta)}$.

\subsection{Main Demonstration}

We first state and prove the following Lemma and then use this result to prove Theorem~\ref{thm:m-regret}.
\begin{theorem}\label{thm:m-regret}
Let Assumption~\ref{ass:distr2} hold. For any $0<\delta \leq 1$ and for $n \geq 4K$, the algorithm MC-UCB launched on a partition $\N_K$ satisfies
%\begin{equation*}
%L_n = \sum_p \frac{w_p^2\si_{p}^2}{T_{p,n}} \leq \frac{\Sigma_n^2}{n} + \frac{4 \Sigma_n\sqrt{2}A}{\sqrt{B}} \frac{K_n^{1/3}}{n^{4/3}} + \frac{12K_n\Sigma_n^2}{n^2}.
%\end{equation*}
\begin{align*}
\E L_n \leq \frac{\Sigma_{\N_K}^2}{n} + 24\sqrt{2} \Sigma_{\N_K}\sqrt{(1 + 3b + 4 f_{\max}^2)}\Big(\frac{f_{\max} +  4}{4}\Big)^{1/3} \frac{K^{1/3}}{n^{4/3}}\sqrt{\log(nK)} + \frac{14K\Sigma_{\N_K}^2}{n^2}.
\end{align*}
\end{theorem}

\begin{proof}

\noindent
{\bf Step~1.~Lower bound of order $\widetilde O(n^{2/3})$.} Let $k$ be the index of an arm such that $T_{k,n} \geq \frac{n}{K}$ (this implies $T_{k,n} \geq 3$ as $n\geq 4K$, and arm $k$ is thus pulled after the initialization) and let $t+1\leq n$ be the last time at which it was pulled~\footnote{Note that such an arm always exists for any possible allocation strategy given the constraint $n=\sum_{q}T_{q,n}$.}, i.e.,~$T_{k,t}=T_{k,n}-1$ and $T_{k,t+1}=T_{k,n}$. From Equation \ref{eq:ucb-maurer-event-sub2} and the fact that $T_{k,n} \geq \frac{n}{K}$, we obtain on $\xi$
\begin{equation}\label{e:meca.lb}
B_{k,t} \leq \frac{w_k}{T_{k,t}}\Bigg(\si_{k} + 2A \sqrt{\frac{1}{T_{k,t}}} \Bigg) \leq \frac{K w_k\Big(\sigma_k + 2A\Big)} {n},
\end{equation}
where the second inequality follows from the facts that $T_{k,t}\geq 1$, $w_k \si_k \leq \Sigma_{\N_K}$, and $w_k \leq \sum_k w_k = 1$. Since at time $t+1$ the arm $k$ has been pulled, then for any arm $q$, we have
\begin{equation}\label{e:meca}
B_{q,t} \leq B_{k,t}.
\end{equation}
From the definition of $B_{q,t}$, and also using the fact that $T_{q,t} \leq T_{q,n}$, we deduce on $\xi$ that
\begin{equation}\label{e:meca.hb}
B_{q,t} \geq \frac{2 A w_q}{T_{q,t}^{3/2}} \geq \frac{2A w_q}{T_{q,n}^{3/2}}\;.
\end{equation}
Combining Equations~\ref{e:meca.lb}--\ref{e:meca.hb}, we obtain on $\xi$
\begin{equation*}
\frac{2A w_q}{T_{q,n}^{3/2}} \leq  \frac{K w_k\Big(\sigma_k + 2A\Big)} {n}.
\end{equation*}
Finally, this implies on $\xi$ that for any $q$ because $w_k = w_q$,
\begin{equation}\label{eq:lb.vloose}
T_{q,n} \geq \Big(\frac{2A}{\si_k + 2A}\frac{n}{K}\Big)^{2/3}.
\end{equation}

This implies that $\forall q, T_{q,n} \geq C \Big(\frac{n}{K}\Big)^{2/3}$ where $C = \Big(\frac{2A}{\max_k \si_k + 2A}\Big)^{2/3}$.

\noindent
{\bf Step~2.~Properties of the algorithm.} We first remind the definition of $B_{q,t+1}$ used in the MC-UCB algorithm
\begin{equation*}
B_{q,t+1} = \frac{w_q}{T_{q,t}} \Bigg( \hsi_{q,t} + A\sqrt{\frac{1}{T_{q,t}}}\Bigg).
\end{equation*}

Using Corollary~\ref{l:bernstein-var-sub} it follows that, on $\xi$
\begin{equation}\label{e:bound.B}
 \frac{w_q\si_q}{T_{q,t}} \leq B_{q,t+1} \leq \frac{w_q}{T_{q,t}}\Bigg(\si_{q} + 2A \sqrt{\frac{1}{T_{q,t}}} \Bigg).
\end{equation}
Let $t+1 \geq 2K+1$ be the time at which an arm $q$ is pulled for the last time, that is $T_{q,t} = T_{q,n}-1$. Note that there is at least one arm such that this happens as $n \geq 4K$. Since at $t+1$ arm $q$ is chosen, then for any other arm $p$, we have
\begin{equation}\label{e:meca2}
B_{p,t+1} \leq B_{q,t+1}\;.
\end{equation}
From Equation~\ref{e:bound.B} and $T_{q,t} = T_{q,n} -1$, we obtain on $\xi$
\begin{equation}\label{e:meca.lb2}
B_{q,t+1} \leq \frac{w_q}{T_{q,t}}\Bigg(\si_{q} + 2A \sqrt{\frac{1}{T_{q,t}}} \Bigg) = \frac{w_q}{T_{q,n}-1}\Bigg(\si_{q} + 2A \sqrt{\frac{1}{T_{q,n}-1}} \Bigg).
\end{equation}
Furthermore, since $T_{p,t} \leq T_{p,n}$, then on $\xi$
\begin{equation}\label{e:meca.hb2}
B_{p,t+1} \geq \frac{w_p\si_{p}}{T_{p,t}} \geq \frac{w_p\si_{p}}{T_{p,n}}.
\end{equation}
Combining Equations~\ref{e:meca2}--\ref{e:meca.hb2}, we obtain on $\xi$
\begin{equation*}
\frac{w_p\si_{p}}{T_{p,n}}(T_{q,n}-1) \leq w_q\Bigg(\si_{q} + 2A \sqrt{\frac{1}{T_{q,n}-1}} \Bigg).
\end{equation*}
Summing over all $q$ such that the previous Equation is verified, i.e. such that $T_{q,n} \geq 3$, on both sides, we obtain on $\xi$

\begin{align*}
\frac{w_p\si_{p}}{T_{p,n}}\sum_{q|T_{q,n} \geq 3}(T_{q,n}-1) \leq \sum_{q|T_{q,n} \geq 3}w_q \Bigg(\si_{q} + 2A \sqrt{\frac{1}{T_{q,n}-1}} \Bigg).
\end{align*}

This implies

\begin{equation}\label{eq:main.lb}
\frac{w_p\si_{p}}{T_{p,n}}(n - 3K) \leq \sum_{q=1}^K w_q\Bigg(\si_{q} + 2A \sqrt{\frac{1}{T_{q,n}-1}} \Bigg).
\end{equation}

\noindent
{\bf Step~3.~Lower bound.} Plugging Equation~\ref{eq:lb.vloose} in Equation~\ref{eq:main.lb},
\begin{align*}
\frac{w_p\si_{p}}{T_{p,n}}(n - 3K) &\leq \sum_q w_q\Bigg(\si_{q} + 2A \sqrt{\frac{1}{T_{q,n} -1}} \Bigg)\\
&\leq \sum_q w_q\Bigg(\si_{q} + 2A \sqrt{\frac{2K^{2/3}}{Cn^{2/3}}} \Bigg)\\
&\leq \Sigma_{\N_K} + \frac{2\sqrt{2}A}{\sqrt{C}} \frac{K^{1/3}}{n^{1/3}},
\end{align*}
on $\xi$, since $T_{q,n} -1 \geq \frac{T_{q,n}}{2}$ (as $T_{q,n}\geq 2$).
Finally as $n \geq 4K$, we obtain on $\xi$ the following bound
\begin{equation}\label{eq:almost-lower-bound}
\frac{w_p\si_{p}}{T_{p,n}} \leq \frac{\Sigma_{\N_K}}{n} + \frac{4\sqrt{2}A}{\sqrt{C}} \frac{K^{1/3}}{n^{4/3}} + \frac{12K\Sigma_{\N_K}}{n^2}.
\end{equation}

\noindent
{\bf Step~4.~Regret.} By summing and using Equation~\ref{eq:almost-lower-bound} which holds for all $p$, we obtain on $\xi$ (with probability $1-\delta$)
\begin{equation*}
L_n = \sum_p \frac{w_p^2\si_{p}^2}{T_{p,n}} \leq \frac{\Sigma_{\N_K}^2}{n} + \frac{4 \Sigma_{\N_K}\sqrt{2}A}{\sqrt{C}} \frac{K^{1/3}}{n^{4/3}} + \frac{12K\Sigma_{\N_K}^2}{n^2}.
\end{equation*}
This implies since $\E L_n = \E [L_n\ind{\xi}] + \E [L_n\ind{\xi^c}]$ and since $\delta = n^{-2}$
\begin{align*}
\E L_n &\leq \frac{\Sigma_{\N_K}^2}{n} + \frac{4 \Sigma_{\N_K}\sqrt{2}A}{\sqrt{C}} \frac{K^{1/3}}{n^{4/3}} + \frac{12K\Sigma_{\N_K}^2}{n^2} + (\sum_{p} w_p^2 \si_p^2) n^{-2}\\
&\leq \frac{\Sigma_{\N_K}^2}{n} + \frac{4 \Sigma_{\N_K}\sqrt{2}A}{\sqrt{C}} \frac{K^{1/3}}{n^{4/3}} + \frac{14K\Sigma_{\N_K}^2}{n^2}.
\end{align*}

Since $\delta = n^{-2}$, we have $A \leq 6\sqrt{(1 + 3b + 4\bar{V})\log(nK)}$ and $C \geq \Big(\frac{4}{f_{\max} +  4}\Big)^{2/3}$, this leads to
\begin{align*}
\E L_n \leq \frac{\Sigma_{\N_K}^2}{n} + 24\sqrt{2} \Sigma_{\N_K}\sqrt{(1 + 3b + 4 f_{\max}^2)}\Big(\frac{f_{\max} +  4}{4}\Big)^{1/3} \frac{K^{1/3}}{n^{4/3}}\sqrt{\log(nK_n)} + \frac{14K\Sigma_{\N_K}^2}{n^2}.
\end{align*}

\end{proof}

\section{Proof of Proposition~\ref{lem:absconv}}\label{s:cvdimd}

\paragraph{Step 1: Expression of the variance of the stratified estimate.}
Note that the samples $f(x) + s(x) \no_t$ where $\no_t \sim \nu_x$ and $\E_{\nu_x}[\no_t] =0$, $\V_{\nu_x}[\no_t] = 1$ the $\no_t$ are independent.
\\
We have
\begin{align*}
\si_k^2 &= \frac{1}{w_k} \int_{\Omega_k} \E_{\nu_x}[(X_x(t) - \mu_k)^2]dx\\
&= \frac{1}{w_k} \int_{\Omega_k} \E_{\nu_x}\Big[(f(x) +s(x)\no_t - \frac{1}{w_k} \int_{\Omega_k} f(u)du)^2\Big]dx\\
&= \frac{1}{w_k} \int_{\Omega_k} \E_{\nu_x}\Big[(f(x) - \frac{1}{w_k} \int_{\Omega_k} f(u)du)^2\Big]dx + \frac{1}{w_k} \int_{\Omega_k} \E_{\nu_x}\Big[s(x)^2\no_t^2\Big]dx\\
&= \frac{1}{w_k} \int_{\Omega_k} \big(f(x) - \frac{1}{w_k} \int_{\Omega_k} f(u)du \big)^2 dx + \frac{1}{w_k} \int_{\Omega_k} s(x)^2dx
\end{align*}

\paragraph{Step 2: Proof for the uniformly continuous functions.}

We first prove the result for a subset of $L_2([0,1]^d)$, namely the set of functions $m$ and $s$ that are uniformly continuous.

\begin{proposition}\label{lem:absconv1}
 If the functions $f$ and $s$ are uniformly continuous and if the strata satisfy the Assumptions of Proposition~\ref{lem:absconv}, we have
\begin{equation*}
 \sum_k w_{k,n}\si_{k,n} - \int_{[0,1]^d} s(x) dx \rightarrow 0
\end{equation*}
\end{proposition}
\begin{proof}

Let $\upsilon >0$. As $s$ and $f$ are uniformly continuous, we know that $\forall x$, $\exists \eta$ such that $|s(x+u) - s(x)|\leq \upsilon$ and $|f(x+u) -f(x)|\leq \upsilon$ where $u \in \B_{2,d}(\eta)$\footnote{We denote by $B_{2,d}(\eta)$ the ball of center $0$ and radius $\eta$ according to the $||.||_2$ norm.}.
\\
By Assumption~AS1, we know that $w_{k,n} \leq \upsilon_n$. Note that the diameter of strata $\Omega_{k,n}$ is smaller than $D(w_{k,n}) \leq D(\upsilon_n)$. Let us choose $n$ big enough, i.e.~such that $D(\upsilon_n) \leq \eta$ and $\upsilon_n \leq \upsilon$.
\\
We have 
\begin{align*}
\si_{k,n}^2 - (\frac{1}{w_{k,n}} \int_{\Omega_{k,n}} s)^2 &=  \frac{1}{w_{k,n}}\int_{\Omega_{k,n}} s^2  - \Big(\frac{1}{w_{k,n}}\int_{ \Omega_{k,n}} s \Big)^2 + \frac{1}{w_{k,n}}\int_{\Omega_{k,n}} \Big( f - \frac{1}{w_{k,n}}\int_{\Omega_{k,n}} f  \Big)^2\\
&= \frac{1}{w_{k,n}}\int_{\Omega_{k,n}} \Big( s - \frac{1}{w_{k,n}}\int_{\Omega_{k,n}} s \Big)^2 + \frac{1}{w_{k,n}}\int_{\Omega_{k,n}} \Big( f - \frac{1}{w_{k,n}}\int_{\Omega_{k,n}} f  \Big)^2 \\
&\leq \upsilon^2 + \upsilon^2 \leq 2\upsilon^2.
\end{align*}

Because of concavity of the square-root function, we get
\begin{align*}
\si_{k,n} - (\frac{1}{w_{k,n}} \int_{\Omega_{k,n}} s) &\leq  \sqrt{2}\upsilon.
\end{align*}

By summing we get
\begin{equation*}
 \sum_k w_{k,n}\si_{k,n} - \int_{[0,1]^d} s  \leq \sqrt{2} \upsilon.
\end{equation*}

\end{proof}

\paragraph{Step 3: Density of uniformly continuous functions in $L_{2}([0,1]^d)$.}

We first remind a property of the functions in $L_{2}([0,1]^d)$.

\begin{proposition}\label{prop:density}
The uniformly continuous functions according to the $||.||_2$ norm are dense in $L_{2}([0,1]^d)$.
\end{proposition}

\begin{proof}
 The result follows directly from the facts that
\begin{itemize}
 \item The continuous functions are dense in $L_{2}(\Omega)$ (Stone-Weierstrass Theorem).
\item The uniformly continuous functions on a compact space $\Omega$ according to the $||.||_2$ norm are dense in the space of continuous functions.
\item $[0,1]^d$ is a compact.
\end{itemize}

\end{proof}
This means that we can approximate with arbitrary precision according to the $||.||_2$ measure on $L_{2}([0,1]^d)$ any function in $L_{2}([0,1]^d)$ by an uniformly continuous function.
\\
Using this proposition, we can prove the following Lemma.
\begin{lemma}\label{lem:unifcont}
For a given $n$ and a given $\upsilon$, there exist two uniformly continuous function $m_{\upsilon}$ and $s_{\upsilon}$ such that:
\begin{equation*}
\Big|\sum_{k=1}^{K_n} w_{k,n}\si_{k,n} - \sum_{k=1}^{K_n} \sqrt{w_{k,n}}\sqrt{ \int_{\Omega_{k,n}} \Big(f_{\upsilon}(x) + \int_{\Omega_{k,n}} f_{\upsilon}(u)du\Big)^2dx   - \frac{1}{w_{k,n}} \int_{\Omega_{k,n}} s_{\upsilon}^2(x) dx}\Big| \leq \upsilon.
\end{equation*}
\end{lemma}
\begin{proof}
Let us fix $n$ and $\upsilon$.
\\
 Let $m_{\upsilon}$ be an uniformly continuous function such that
\begin{equation*}
 \int_{\Omega} (f(x)-f_{\upsilon}(x))^2 dx \leq \min_k(w_{k,n})\frac{\upsilon}{2}, 
\end{equation*}
and $s_{\upsilon}$ be an uniformly continuous function such that
\begin{equation*}
 \int_{\Omega} (s(x)-s_{\upsilon}(x))^2 dx \leq \min_k(w_{k,n})\frac{\upsilon}{2}.
\end{equation*}
It is possible because of $w_{k,n} >0$ and because the uniformly continuous functions are dense in $L_2([0,1]^d)$ by Proposition~\ref{prop:density}.
\\
Note that we thus have
\begin{equation*}
\frac{1}{w_{k,n}} \int_{\Omega_{k,n}} (f(x)-f_{\upsilon}(x))^2 dx \leq \frac{\upsilon}{2}, 
\end{equation*}
and
\begin{equation*}
 \frac{1}{w_{k,n}}  \int_{\Omega_{k,n}} (s(x)-s_{\upsilon}(x))^2 dx \leq \frac{\upsilon}{2}.
\end{equation*}
Note also that $\frac{1}{w_{k,n}}  \int_{\Omega_{k,n}} (s(x)-s_{\upsilon}(x))^2 dx \geq  \Big|\frac{1}{w_{k,n}}  \int_{\Omega_{k,n}} s(x)^2 dx - \frac{1}{w_{k,n}}  \int_{\Omega_{k,n}} s_{\upsilon}(x)^2 dx \Big|$.
\\
Simple triangle inequality leads to
\begin{equation*}
\Big| \frac{1}{w_{k,n}} \int_{\Omega_{k,n}} (f(x) - \frac{1}{w_{k,n}} \int_{\Omega_{k,n}} f(u)du)^2 dx - \frac{1}{w_{k,n}} \int_{\Omega_{k,n}} (f_{\upsilon}(x) - \frac{1}{w_{k,n}} \int_{\Omega_{k,n}} f_{\upsilon}(u) du)^2 dx \Big| \leq \frac{\upsilon}{2}.
\end{equation*}
Now note that as $\si_{k,n}^2 = \frac{1}{w_{k,n}} \int_{\Omega_{k,n}} (f(x) - \frac{1}{w_{k,n}} \int_{\Omega_{k,n}} f(u)du)^2 dx + \frac{1}{w_{k,n}} \int_{\Omega_{k,n}} s(x)^2 dx$, we know that the variance of the function on strata $\Omega_{k,n}$ is arbitrarily close to the variance of its approximation.
\\
By convexity, one gets
\begin{equation*}
\Big| \si_{k,n} - \sqrt{\frac{1}{w_{k,n}} \int_{\Omega_{k,n}} \Big(f_{\upsilon}(x) - \frac{1}{w_{k,n}} \int_{\Omega_{k,n}} f_{\upsilon}(u)du\Big)^2dx  + \frac{1}{w_{k,n}} \int_{\Omega_{k,n}} s_{\upsilon}^2(x) dx}\Big| \leq \upsilon.
\end{equation*} 
And finally, by summing
\begin{equation*}
\Big|\sum_{k=1}^{K_n} w_{k,n}\si_{k,n} - \sum_{k=1}^{K_n} \sqrt{w_{k,n}}\sqrt{ \int_{\Omega_{k,n}} \Big(f_{\upsilon}(x) + \int_{\Omega_{k,n}} f_{\upsilon}(u)du\Big)^2dx   - \frac{1}{w_{k,n}} \int_{\Omega_{k,n}} s_{\upsilon}^2(x) dx}\Big| \leq \upsilon.
\end{equation*}
\end{proof}

\paragraph{Step 4: Combination of all the preliminary results to finish the proof.}
Finally, we finish the demonstration of Proposition~\ref{lem:absconv}.

 Let $\upsilon>0$ and $f_{\upsilon}$ and $s_{\upsilon}$ be as in Lemma~\ref{lem:unifcont}.
\\
We know that 
\begin{equation*}
\Big|\sum_{k=1}^{K_n} w_{k,n}\si_{k,n} - \sum_{k=1}^{K_n} \sqrt{w_{k,n}}\sqrt{ \int_{\Omega_{k,n}} \Big(f_{\upsilon}(x) + \int_{\Omega_{k,n}} f_{\upsilon}(u)du\Big)^2dx   - \frac{1}{w_{k,n}} \int_{\Omega_{k,n}} s_{\upsilon}^2(x) dx}\Big| \leq \upsilon,
\end{equation*}
and also that
\begin{equation*}
 \int_{\Omega} (s(x)-s_{\upsilon}(x))^2 dx \leq \min_k(w_{k,n})\frac{\upsilon}{2} \leq \frac{\upsilon}{2}.
\end{equation*}
Note that by Cauchy-Schwartz:
\begin{equation*}
\int_{\Omega} |s(x)-s_{\upsilon}(x)| dx \leq  \sqrt{\int_{\Omega} (s(x)-s_{\upsilon}(x))^2 dx}  \leq \sqrt{\frac{\upsilon}{2}}.
\end{equation*}
Note also that Proposition~\ref{lem:absconv1} tells us that $\exists n$ such that
\begin{equation*}
 \sum_{k=1}^{K_n} \sqrt{w_{k,n}}\sqrt{\int_{\Omega_{k,n}} \Big(f_{\upsilon}(x) - \frac{1}{w_{k,n}} \int_{\Omega_{k,n}} f_{\upsilon}(u)du\Big)^2dx   +  \int_{\Omega_{k,n}} s_{\upsilon}^2(x) dx} - \int_{[0,1]^d} s_{\upsilon}(x)dx \leq \upsilon.
\end{equation*}
When combining all those results, one gets the desired result.

Note finally that if we choose the strata as being small boxes of size $\frac{1}{K}$ and side $(\frac{1}{K})^{1/d}$, then the assumptions of Proposition~\ref{lem:absconv} is verified.

\section{Proof of Proposition~\ref{prop:noisy}}\label{ss:ratelip}

Note first that
\begin{align*}
 \sigma_{k}^2 = \frac{1}{w_{k}}\int_{\Omega_{k}} \Big(f(x) - \frac{1}{w_{k}} \int_{\Omega_{k}} f(u)du \Big)^2dx + \frac{1}{w_{k}} \int_{\Omega_{k}} s^2(x)dx.
\end{align*}

\paragraph{The term in $f$}
As the function $f$ is $(\alpha,M)-$ H\"older, we know that $\forall(x,y)  \in \Omega, |f(x) - f(y)| \leq M||x-y||_2^{\alpha}$.
\\
Using that we get
\begin{align*}
\frac{1}{w_{k}}\int_{\Omega_{k}} \Big(f(x) - \frac{1}{w_{k}} \int_{\Omega_{k}} f(u)du \Big)^2dx &\leq M^2D(\Omega_k)^{2\alpha}\\
&\leq M^2 d (\frac{1}{K})^{2\alpha/d}.
\end{align*}

\paragraph{The term in $s$}
As the function $s$ is $(\alpha,M)-$ H\"older, we know that $\forall(x,y)  \in \Omega, |s(x) - s(y)| \leq M||x-y||_2^{\alpha}$.
\begin{align*}
 \frac{1}{w_{k}} \int_{\Omega_{k}} s^2(x)dx - \big(\frac{1}{w_{k}} \int_{\Omega_{k}} s(u)du\big)^2 = \frac{1}{w_{k}} \int_{\Omega_{k}} \big(s(x) - \frac{1}{w_{k}} \int_{\Omega_{k}} s(u)du\big)^2 dx &\leq M^2D(\Omega_{k})^{2\alpha}\\
&\leq M^2  d (\frac{1}{K})^{2\alpha/d}.
\end{align*}

\paragraph{Finally...}
By combining those two results
\begin{align*}
w_{k}\sigma_{k} - \int_{\Omega_{k}} s(x)dx &\leq  w_{k} \sqrt{\si_{k}^2 - \big(\frac{1}{w_{k}} \int_{\Omega_{k}} s(x)dx\big)^2}\\
&\leq w_{k} \sqrt{M^2 d (\frac{1}{K})^{2\alpha/d} + M^2 d (\frac{1}{K})^{2\alpha/d}}.
\end{align*}
By summing over all the strata, one obtains
\begin{equation*}
 \Sigma_{\N_K} - \int_{[0,1]^d} s(x) dx \leq \sqrt{2d} M  (\frac{1}{K})^{\alpha/d}.
\end{equation*}

\section{Lower bound}\label{app:lowbound}

Let us write the proof of the lower bound using the terminology of multi-armed bandits. Each arm $k$ represents a stratum and the distribution associated to this arm is defined as the distribution of the noisy samples of the function collected when sampling uniformly on the strata.

Let us choose $\mu<1/2$ and $\al = \frac{\mu}{2}$. Consider $2K$ Bernoulli bandits (i.e., $2K$ strata where the samples follow Bernoulli distributions) where the $K$ first bandits have parameter $(\mu_k)_{1\leq k\leq K}$ and the $K$ last ones have parameter $1/2$. The $\mu_k$ take values in $\{\mu-\al,\mu,\mu+\al\}$.

Define $\sigma^2 = \mu(1-\mu)$ the variance of a Bernoulli of parameter $\mu$, and is such that $\sqrt{\frac{1}{2} \mu} \leq \sigma \leq \sqrt{\mu}$. We wite $\sigma_{-\al}$ and $\sigma_{+\al}$ the two other standard deviations, and notice that $\frac{1}{2}\sqrt{\mu} \leq \sigma_{-\al} \leq \sqrt{\mu}$, and $\sqrt{\frac{1}{2} \mu} \leq \sigma_{+\al} \leq \sqrt{\mu}$.

We consider the $2^K$ bandit environments $M(\upsilon)$ (characterized by  $\upsilon=(\upsilon_k)_{1\leq k \leq K} \in \{-1,+1\}^K$) defined by $(\mu_k = \mu + \upsilon_k \al)_{1\leq k\leq K}$. We write $\P_{\upsilon}$ the probability with respect to the environment $M(\upsilon)$ at time $n$. We also write $M(\sigma)$ the environment defined by all $K$ first arms having a parameter $\sigma$, and write $\P_{\sigma}$ the associated probability at time $n$.

The optimal oracle allocation for environment $M(\upsilon)$ is to play arm $k \leq K$, $t_k(\upsilon)=\frac{\sigma_{\upsilon_k \al}}{\sum_{i=1}^K \sigma_{\upsilon_i \al} +K/2} n$ times and arm $k > K$, $t_{k}(\upsilon)=\frac{1/2}{\sum_{i=1}^K \sigma_{\upsilon_i \al} +K/2} n$ times. The corresponding quadratic error of the resulting estimate is $l(\upsilon) = \frac{(\sum_{i=1}^K \sigma_{\upsilon_i \al} +K/2)^2}{(2K)^2n}$. For the environment $M(\sigma)$, the optimal oracle allocation is to play arm $k \leq K$, $t(\sigma)=\frac{\sigma}{K\sigma +K/2} n$ times (and arm $k > K$, $t_2(\sigma)=\frac{1/2}{K\sigma +K/2} n$ times).

Consider deterministic algorithms first (extension to randomized algorithms will be discussed later). An algorithm is a set (for all $t=1$ to $n-1$) of mappings from any sequence $(r_1,\dots,r_t)\in\{0,1\}$ of $t$ observed samples (where $r_s\in\{0,1\}$ is the sample observed at the $s$-th round) to the choice of an arm $I_{t+1}\in\{1,\ldots,2K\}$.  Write $T_k(r_1,\dots,r_{n})$ the (random variable) corresponding to the number of pulls of arm $k$ up to time $n$. We thus have $n = \sum_{k=1}^{2K} T_k$.

Now, consider the set of algorithms that know that the $K$ first arms have parameter $\mu_k \in \{\mu-\al,\mu,\mu+\al\}$, and that also know that the $K$ last arms have their parameters in $\{1/4,3/4\}$. 
Given this knowledge, an optimal algorithm will not pull any arm $k\leq K$ more than $\Big(\frac{\sigma_{+\al}}{K\sigma_{-\al} + \sqrt{3}K/4}\Big) n$ times. Indeed, the optimal oracle allocation in \textit{all} such environments allocates less than $\Big(\frac{\sigma_{+\al}}{K\sigma_{-\al} + \sqrt{3}K/4}\Big) n$ samples to each arm $k\leq K$. In addition, since the samples of all arms are independent, a sample collected from arm $k$ does not provide any information about the relative allocations among the other arms. Thus, once an arm has been pulled as many times as recommended by the optimal oracle strategy, there is no need to allocate more samples to that arm. Writing $\mathbb A$ the class of all algorithms that do not know the set of possible environments, $\mathbb A_{\upsilon}$ the class of algorithms that know the set of possible environments $M(\upsilon)$ and $\mathbb A_{opt}$ the subclass of $\mathbb A_{\upsilon}$ that pull all arms $k \leq K$ less than $\Big(\frac{\sigma_{+\al}}{K\sigma_{-\al} + \sqrt{3}K/4}\Big) n$ times, we have
$$ \inf_{\mathbb A} \sup_{M(\upsilon)} \E R_n \geq \inf_{\mathbb A_{\upsilon}} \sup_{M(\upsilon)} \E R_n = \inf_{\mathbb A_{opt}} \sup_{M(\upsilon)} \E R_n,$$
where the first inequality comes from the fact that algorithms in $\mathbb A_{\upsilon}$ possess more information than those in $\mathbb A$, which they can use or not. Thus $\mathbb A\subset \mathbb A_{\upsilon}$.

Now for any $\upsilon = (\upsilon_1,\ldots,\upsilon_K)$, define the events 
\begin{align*}
\Omega_{\upsilon} = \{\omega : \forall \U \subset \{1,\ldots,K\} : |\U| \leq \frac{K}{3} \mbox{ and } \forall k \in \U^c, \upsilon_k T_k \geq \upsilon_k t(\sigma) \}.
\end{align*}

Note that by definition
\begin{align*}
\Omega_{\upsilon} = \bigcup_{p=1}^{\frac{K}{3}}\bigcup_{\U \subset \{1,\ldots,K\} : |\U| = p} \Bigg\{\Big\{\bigcap_{ k \in \U} \{\upsilon_k T_k < \upsilon_k t(\sigma) \} \Big\} \bigcap \Big\{\bigcap_{ k \in \U^C} \{\upsilon_k T_k \geq \upsilon_k t(\sigma) \} \Big\}\Bigg\}.
\end{align*}

By the sub-additivity of the probabilities, we have
\begin{align*}
\P_{\sigma}(\Omega_{\upsilon}) \leq \sum_{p=1}^{\frac{K}{3}}\sum_{\U \subset \{1,\ldots,K\} : |\U| = p} \P\Bigg[\Bigg\{\Big\{\bigcap_{ k \in \U} \{\upsilon_k T_k < \upsilon_k t(\sigma) \} \Big\} \bigcap \Big\{\bigcap_{ k \in \U^C} \{\upsilon_k T_k \geq \upsilon_k t(\sigma) \} \Big\}\Bigg\}\Bigg].
\end{align*}

The events $\Bigg\{\Big\{\bigcap_{ k \in \U} \{\upsilon_k T_k < \upsilon_k t(\sigma) \} \Big\} \bigcap \Big\{\bigcap_{ k \in \U^C} \{\upsilon_k T_k \geq \upsilon t(\sigma) \} \Big\}\Bigg\}$ are disjoint for different $\upsilon$, and form a partition of the space, thus $\sum_{\upsilon} \P_{\sigma} \Bigg[\Bigg\{\Big\{\bigcap_{ k \in \U} \{\upsilon_k T_k < \upsilon_k t(\sigma) \} \Big\} \bigcap \Big\{\bigcap_{ k \in \U^C} \{\upsilon T_k \geq \upsilon_k t(\sigma) \} \Big\}\Bigg\} \Bigg]=1$.

We deduce that
\begin{align*}
\sum_{\upsilon}\P_{\sigma}(\Omega_{\upsilon}) &\leq \sum_{\upsilon}\sum_{p=1}^{\frac{K}{3}}\sum_{\U \subset \{1,\ldots,K\} : |\U| = p} \P_{\sigma} \Bigg[\Bigg\{\Big\{\bigcap_{ k \in \U} \{\upsilon T_k < \upsilon_k t(\sigma) \} \Big\} \bigcap \Big\{\bigcap_{ k \in \U^C} \{\upsilon_k T_k \geq \upsilon_k t(\sigma) \} \Big\}\Bigg\} \Bigg]\\
&= \sum_{p=1}^{\frac{K}{3}}\sum_{\U \subset \{1,\ldots,K\} : |\U| = p} \sum_{\upsilon} \Bigg[\Bigg\{\Big\{\bigcap_{ k \in \U} \{\upsilon_k T_k < \upsilon_k t(\sigma) \} \Big\} \bigcap \Big\{\bigcap_{ k \in \U^C} \{\upsilon T_k \geq \upsilon_k t(\sigma) \} \Big\}\Bigg\} \Bigg]\\
&= \sum_{p=1}^{\frac{K}{3}}\sum_{\U \subset \{1,\ldots,K\} : |\U| = p} 1\\
&= \sum_{p=1}^{\frac{K}{3}} \left( \begin{array}{c} K\\ p \end{array} \right).
\end{align*}

%Note that for environment $M(\C,\upsilon)$, there is $\{ \Delta T \leq \big(\frac{K}{3} \} \subset \Omega_{\C,\upsilon}$.

Since there are $2^K$ environments $\upsilon$, we have
$$
 \min_{\upsilon} \P_{\sigma} (\Omega_{\upsilon}) \leq \frac{1}{2^{K}}\sum_{\upsilon}\P_{\sigma}(\Omega_{\upsilon})
\leq \frac{1}{2^{K}}\sum_{p=1}^{\frac{K}{3}} \left( \begin{array}{c} K\\ p \end{array} \right).
$$

Note that $\frac{1}{2^K}\sum_{p=1}^{\frac{K}{3}} \left( \begin{array}{c} K\\ p \end{array} \right) = \P(\sum_{k=1}^{K} X_k \leq \frac{K}{3})$ where $(X_1,\ldots,X_K)$ are $K$ independent Bernoulli random variables of parameter $1/2$. By Chernoff-Hoeffding's inequality, we have $\P(\sum_{k=1}^{K} X_k \leq \frac{K}{3}) = \P(\frac{1}{K}\sum_{k=1}^{K} X_k - \frac{1}{2}\leq \frac{K}{6}) \leq \exp(-K/72)$.
Thus there exists $\upsilon_{\min}$ such that $\P_{\sigma} (\Omega_{\upsilon_{\min}}) \leq \exp(-K/72)$.

Let us write $p = \P_{\upsilon_{\min}} (\Omega_{\upsilon_{\min}})$ and $p_{\sigma} = \P_{\sigma} (\Omega_{\upsilon_{\min}})$. Let $kl(a,b)=a\log(\frac ab)+(1-a)\log(\frac{1-a}{1-b})$ denote the KL for Bernoulli distributions with parameters $a$ and $b$. Note that because $\forall \Omega$, $KL(\P_{\upsilon_{\min}}(.|\Omega), \P_{\sigma}(.|\Omega)) \geq 0$, we have
\begin{align*}
  kl(p,p_{\sigma}) \leq KL(\P_{\upsilon_{\min}}, \P_{\sigma}).
\end{align*}

From that we deduce that $p(\log(p) - \log(p_{\si})) + (1-p) (\log(1-p) - \log(1-p_{\si})) \leq KL(\P_{\upsilon_{\min}}, \P_{\sigma})$, which leads to
\begin{align}\label{eq:maxcmin}
 p  \leq \max( \frac{36}{K}\Big(KL(\P_{\upsilon_{\min}}, \P_{\sigma})  \Big), \exp(-K/72)).
\end{align}

Let us now consider any environment $(\upsilon)$. Let $R_t = (r_1,\ldots,r_t)$ be the sequence of observations, and let $\P^t_{\upsilon}$ be the law of $R_t$ for environment $M(\upsilon)$. Note first that $\P_{\upsilon} = \P^n_{\upsilon}$. Adapting the chain rule for Kullback-Leibler divergence, we get
\begin{align*}
 &KL(\P^n_{\upsilon}, \P^n_{\sigma})\\ 
&= KL(\P^1_{\upsilon}, \P^1_{\sigma}) + \sum_{t=2}^n \sum_{R_{t-1}} \P_{\upsilon}^{t-1} (R_{t-1}) KL(\P^t_{\upsilon}(.|R_{t-1}), \P^t_{\sigma}(.|R_t))\\
&= KL(\P^1_{\sigma}, \P^1_{\upsilon}) + \sum_{t=2}^n\Big[ \sum_{R_{t-1}| \upsilon_{I_t} = +1 } \P_{\sigma}^{t-1} (R_{t-1}) kl(\mu +\alpha, \mu) + \sum_{R_{t-1}| \upsilon_{I_t} = -1} \P_{\sigma}^{t-1} (R_{t-1}) kl(\mu-\al , \mu) \Big] \\
&= kl(\mu - \al,\mu) \E_{\upsilon} [\sum_{k : \upsilon_k = -1} T_k] + kl(\mu + \al,\mu) \E_{\upsilon} [\sum_{k : \upsilon_k = +1} T_k].
\end{align*}

We thus have, using the property that $kl(a,b)\leq \frac{(a-b)^2}{b(1-b)}$,
\begin{align*}
KL(\P_{\upsilon}, \P_{\sigma}) &= kl(\mu - \al,\mu) \E_{\upsilon} [\sum_{k : \upsilon_k = -1} T_k] + kl(\mu + \al,\mu) \E_{\upsilon} [\sum_{k : \upsilon_k = +1} T_k]\\
&\leq \E_{\sigma} [\sum_{k\leq K} T_k] \frac{\al^2}{\mu(1-\mu)}\\
&= E_{\sigma}[\sum_{k\leq K} T_k] \frac{\al^2}{\sigma^2}.
\end{align*}

Note that for an algorithm in $\mathbb A_{opt}$, we have $\sum_{k=1}^K T_k \leq T_k \leq K\Big(\frac{\sigma_{+\al}}{K\sigma_{-\al} + \sqrt{3}K/4}\Big) n$. 
Since $\alpha = \frac{\mu}{2}$ and $0 <\mu \leq \frac{1}{2}$ we have
\begin{align*}
KL(\P_{\upsilon}, \P_{\sigma}) &\leq \Big(K \frac{\sigma_{+\al}}{K\sigma_{-\al} + \sqrt{3}K/4}\Big) \frac{\al^2}{\sigma^2} n\\
&\leq 4 \sigma_{+\al}  \frac{\al^2}{\sigma^2} n\\
&\leq 8 \frac{\al^2}{\sigma} n,
\end{align*}

We thus deduce using Equation~\ref{eq:maxcmin}
\begin{align*}
\P_{\upsilon_{\min}}(\Omega_{\upsilon_{\min}}) = p  &\leq \max( \frac{18}{K}\Big(KL(\P_{\upsilon_{\min}}, \P_{\sigma}) \Big), \exp(-K/72))\\
&\leq \frac{144}{K} \frac{\al^2}{\sigma} n.
\end{align*}

% \begin{align*}
%  \Bigg|\P_{\C_{\min},\upsilon_{\min}}(\Omega_{\C_{\min},\upsilon_{\min}}) - \P_{\sigma}(\Omega_{\C_{\min},\epsilon_{\min}})  \Bigg| &\leq  \sqrt{\frac{1}{2K}KL(\P_{\sigma}, \P_{\C_{\min},\epsilon_{\min}}) + \frac{\log(K)}{K^2}}\\
% &\leq  \sqrt{\frac{1}{2K}C \sigma_{+\al}  \frac{\al^2}{\sigma^2} n + \frac{\log(K)}{K^2}}
% \end{align*}
% 
% Thus 
% 
% \begin{align*}
% \min\Big(\P_{\sigma_{-\al},\ldots, \sigma_{-\al}}[T \leq Kt(\sigma)] , \P_{\sigma_{+\al},\ldots, \sigma_{+\al}}[T > Kt(\sigma)]\Big)&\leq \frac 12 +  \frac{2\al}{\sqrt{\sigma}} \sqrt{Kn}.
% \end{align*}

Now choose $\sigma \leq \frac{1}{7}(\frac{K}{n})^{1/3}$ (as $\alpha = \frac{\mu}{2} = \frac{\sigma^2}{2}$). Note that this implies that $\P_{\upsilon_{\min}}(\Omega_{\upsilon_{\min}}) \leq \frac{1}{2}$.

Let $\omega \in \Omega_{\upsilon_{\min}}^c$. We know that for $\omega$, there are at least $\frac{K}{3}$ arms among the $K$ first which are not pulled correctly: either $\frac{K}{6}$ arms among the arms with parameter $\mu-\al$ or among the arms with parameter $\mu+\al$ are not pulled correctly. Assume that for this fixed $\omega$, there are $\frac{K}{6}$ arms among the arms with parameter $\mu-\al$ which are not pulled correctly. Let $\U(\omega)$ be this subset of arms.

We write $\Delta T = \sum_{k \in \U} T_k - \frac{K}{6} t(\sigma_{-\al})$ the number of times those arms are over pulled. Note that on $\omega$ we have $\Delta T \geq \frac{K}{6} t(\sigma) - t(\sigma_{-\al})$. We have
\begin{align*}
\Delta T = \frac{K}{6}t(\sigma) - \frac{K}{6}t(\sigma_{-\alpha}) &= \frac{1}{6} \frac{K\sigma}{K\sigma + K/2}n - \frac{1}{6}\frac{K\sigma_{-\al}}{\sum_{i=1}^K \sigma_{\upsilon_i \al} + K/2}n\\ 
&\geq \frac{1}{6}\frac{K\sigma}{K\sigma + K/2}n- \frac{1}{6}\frac{K\sigma/\sqrt{2}}{\sqrt{3}K\sigma/\sqrt{2} + K/2}n\\
&\geq \frac{1}{6}\frac{1}{K\sigma + K/2}\frac{1}{\sqrt{3}K\sigma/\sqrt{2} + K/2} \Big(K^2 \sigma/2 - K^2 \sigma/2\sqrt{2}\Big) n \\
&\geq \frac{1}{2} (1 - 1/\sqrt{2}) \sigma n\\
&\geq \frac{1}{35} K^{1/3} n^{2/3}
\end{align*}

Thus on $\omega$, the regret is such that
\begin{align*}
 R_{n, \upsilon_{\min}}(\omega)
&\geq  \sum_{k=1}^{3K} \frac{w_k^2 \si_k^2}{T_k(\omega)} - \frac{1}{(2K)^2} \frac{(\sum_{i=1}^K \sigma_{\upsilon_i \al} + K/2 \big)^2}{n}\\
&\geq  \sum_{k \in \U(\omega)} \frac{w_k^2 \si_k^2}{T_k(\omega)} + \sum_{k \in \U(\omega)^C} \frac{w_k^2 \si_k^2}{T_k(\omega)} - \frac{1}{(2K)^2} \frac{(\sum_{i=1}^K \sigma_{\upsilon_i \al} + K/2 \big)^2}{n}\\
&\geq \frac{1}{K^2} \frac{K}{6}\frac{ \si_{-\al}^2}{t_k(\sigma_{-\al}) + 6\Delta T/K} + \frac{\big(\sum_{i=1}^K \sigma_{\upsilon_i \al} - K\sigma_{-\al}/6 + K/2\big)^2}{(2K-K/6)^2 (n - \Delta T)} - \frac{1}{(2K)^2} \frac{(\sum_{i=1}^K \sigma_{\upsilon_i \al} + K/2 \big)^2}{n}\\
&\geq \frac{1}{(2K)^2} \frac{\big(\sum_{i=1}^K \sigma_{\upsilon_i \al} + K/2 \big)^2}{n}  \frac{1 + \Big( \frac{\big(\sum_{i=1}^K \sigma_{\upsilon_i \al}  + K/2 \big)\Delta T}{\big( K\sigma_{-\al}/6  \big)n} - \frac{\big(\sum_{i=1}^K \sigma_{\upsilon_i \al}  + K/2 \big)\Delta T}{\big(\sum_{i=1}^K \sigma_{\upsilon_i \al} - K\sigma_{-\al}/6  + K/2\big)n}\Big)}{\Big(1 + \frac{6 \Delta T \big(\sum_{i=1}^K \sigma_{\upsilon_i \al}  + K/2 \big)}{K\si_{-\al}n} \Big) \Big(1 - \frac{\big(\sum_{i=1}^K \sigma_{\upsilon_i \al}  + K/2 \big)\Delta T}{\big(\sum_{i=1}^K \sigma_{\upsilon_i \al} - K\sigma_{-\al}/6  + K/2\big)n}\Big)} \\
 &- \frac{1}{(2K)^2} \frac{(\sum_{i=1}^K \sigma_{\upsilon_i \al} + K/2 \big)^2}{n}\\
&\geq \frac{1}{(2K)^2} \frac{(\sum_{i=1}^K \sigma_{\upsilon_i \al} + K/2 \big)^2}{n} \frac{\Big( \frac{\big(\sum_{i=1}^K \sigma_{\upsilon_i \al}  + K/2 \big)\Delta T}{\big(\sum_{i=1}^K \sigma_{\upsilon_i \al} - K\sigma_{-\al}/6  + K/2\big)n} \Big)\Big( \frac{\big(\sum_{i=1}^K \sigma_{\upsilon_i \al}  + K/2 \big)\Delta T}{\big( K\sigma_{-\al}/6 \big)n} \Big)}{\Big(1 + \frac{6 \Delta T \big(\sum_{i=1}^K \sigma_{\upsilon_i \al}  + K/2 \big)}{K\si_{-\al}n} \Big) \Big(1 - \frac{\big(\sum_{i=1}^K \sigma_{\upsilon_i \al}  + K/2 \big)\Delta T}{\big(\sum_{i=1}^K \sigma_{\upsilon_i \al} - K\sigma_{-\al}/6  + K/2\big)n}\Big)}\\
 &\geq C \frac{(\Delta T)^2}{n^3 \sigma}\\
&\geq C \frac{K^{1/3}}{n^{4/3}},
\end{align*}
where $C$ is a numerical constant. Note that for events $\omega$ where there are $\frac{K}{6}$ arms among the arms with parameter $\mu+\al$ which are not pulled correctly, the same result holds.

Note finally that $\P(\Omega_{\upsilon_{\min}}^c) \geq 1/2$. We thus have that the regret is bigger than
\begin{align*}
 \E R_{n, \upsilon_{\min}}
&\geq  \sum_{\omega \in \Omega_{\upsilon_{\min}}^c}  R_{n, \upsilon_{\min}}(\omega) \P_{\upsilon_{\min}} (\omega)\\
&\geq \sum_{\omega \in \Omega_{\upsilon_{\min}}^c} C \frac{K^{1/3}}{n^{4/3}} \P_{\upsilon_{\min}} (\omega)\\
&\geq \frac{1}{2} C \frac{K^{1/3}}{n^{4/3}},
\end{align*}
which proves the lower bound for deterministic algorithms.  Now the extension to randomized algorithms is straightforward: any randomized algorithm can be seen as a static (i.e., does not depend on samples) mixture of deterministic algorithms (which can be defined before the game starts). Each deterministic algorithm satisfies the lower bound above in expectation, thus any static mixture does so too.

\section{Large deviation inequalities for independent sub-Gaussian random variables}\label{s:tools}

We first state Bernstein inequality for large deviations of independent random variables around their mean.

\begin{lemma}\label{lem:bernstein}
 Let $(X_1,\ldots,X_n)$ be $n$ independent random variables of mean $(\mu_1,\ldots,\mu_n)$ and of variance $(\si_1^2,\ldots,\si_n^2)$. Assume that there exists $b>0$ such that for any $\lambda < \frac{1}{b}$, for any $i \leq n$, it holds that $\E\Big[ \exp(\lambda (X_i-\mu_i)) \Big] \leq \exp\Big( \frac{\lambda^2 \sigma_i^2}{2(1 - \lambda b)}\Big)$. Then with probability $1-\delta$
\begin{equation*}
|\frac{1}{n}\sum_{i=1}^nX_i - \frac{1}{n}\sum_{i=1}^n \mu_i| \leq \sqrt{\frac{2(\frac{1}{n}\sum_{i=1}^n \si_i^2)\log(2/\delta)}{n}} + \frac{b\log(2/\delta)}{n}.
\end{equation*}
\end{lemma}
\begin{proof}
If the assumptions of Lemma~\ref{lem:bernstein} are verified, then
\begin{eqnarray*}
 \P \Big( \sum_{i=1}^n X_i - \sum_{i=1}^n \mu_i \geq n \upsilon \Big) &=  \P \Bigg[ \exp\Big(\lambda(\sum_{i=1}^n X_i - \sum_{i=1}^n \mu_i)\Big) \geq \exp(n\lambda \upsilon) \Bigg]\\
&\leq \E\Bigg[\frac{\exp\Big(\lambda(\sum_{i=1}^n X_i - \sum_{i=1}^n \mu_i)\Big)}{\exp(n\lambda \upsilon)}\Bigg]\\
&\leq \prod_{i=1}^n \E\Bigg[\frac{\exp\Big(\lambda(X_i - \mu_i)\Big)}{\exp(\lambda \upsilon)}\Bigg]\\
&\leq \exp(\frac{\lambda^2}{2}\sum_{i=1}^n \frac{\si_i^2}{2(1-\lambda b)} - n\lambda\upsilon).
\end{eqnarray*}

By setting $\lambda = \frac{n\upsilon}{\sum_{i=1}^n\sigma_i^2+bn\upsilon}$ we obtain
\begin{equation*}
 \P \Big(  \sum_{i=1}^n X_i -  \sum_{i=1}^n\mu_i \geq  n\upsilon \Big) \leq \exp(-\frac{n^2\upsilon^2}{2(\sum_{i=1}^n\si_i^2+bn\upsilon)}).
\end{equation*}

By an union bound we obtain
\begin{equation*}
 \P \Big(  |\sum_{i=1}^nX_i - \sum_{i=1}^n \mu_i| \geq  n\upsilon \Big) \leq 2\exp(-\frac{n^2\upsilon^2}{2(\sum_{i=1}^n\si_i^2+bn\upsilon)}).
\end{equation*}

This means that with probability $1-\delta$,
\begin{equation*}
|\frac{1}{n}\sum_{i=1}^nX_i - \frac{1}{n}\sum_{i=1}^n \mu_i| \leq \sqrt{\frac{2(\frac{1}{n}\sum_{i=1}^n \si_i^2)\log(2/\delta)}{n}} + \frac{b\log(2/\delta)}{n}.
\end{equation*}
\end{proof}

We also state the following Lemma on large deviations for the variance of independent random variables.

\begin{lemma}\label{ss:variance}
 Let $(X_1,\ldots,X_n)$ be $n$ independent random variables of mean $(\mu_1,\ldots,\mu_n)$ and of variance $(\si_1^2,\ldots,\si_n^2)$. Assume that there exists $b>0$ such that for any $\lambda < \frac{1}{b}$, for any $i \leq n$, it holds that $\E\Big[ \exp(\lambda (X_i-\mu_i)) \Big] \leq \exp\Big( \frac{\lambda^2 \sigma_i^2}{2(1 - \lambda b)}\Big)$ and also $\E\Big[ \exp(\lambda (X_i-\mu_i)^2 - \lambda \si_i^2) \Big] \leq \exp\Big( \frac{\lambda^2 \si_i^2}{2(1 - \lambda b)}\Big)$.

Let $V=\frac{1}{n}\sum_i (\mu_i - \frac{1}{n}\sum_i \mu_i)^2 + \frac{1}{n} \sum_n \si_i^2$ be the variance of a sample chosen uniformly at random among the $n$ distributions, and $\hat V = \frac{1}{n} \sum_{i=1}^n \big(X_i - \frac{1}{n}\sum_{j=1}^n X_j \big)^2$ the corresponding empirical variance. Then with probability $1-\delta$,
\begin{equation*}
|\sqrt{\hat{V}} - \sqrt{V}| \leq  2\sqrt{\frac{(1 + 3b + 4V)\log(2/\delta)}{n}}.
\end{equation*}
\end{lemma}

\begin{proof}
By decomposing the estimate of the empirical variance in bias and variance, we obtain with probability $1-\delta$
\begin{align*}
\hat{V} =& \frac{1}{n}\sum_i (X_i - \frac{1}{n}\sum_j \mu_j)^2 - (\frac{1}{n}\sum_i X_i -\frac{1}{n} \sum_i \mu_i)^2\\
=& \frac{1}{n}\sum_i (X_i -  \mu_i)^2 + 2 \frac{1}{n}\sum_i (X_i -  \mu_i)\frac{1}{n}\sum_i (\mu_i - \frac{1}{n}\sum_j \mu_j)\\
&+  \frac{1}{n}\sum_i (\mu_i - \frac{1}{n}\sum_j \mu_j)^2  - (\frac{1}{n}\sum_i X_i -\frac{1}{n} \sum_i \mu_i)^2\\
=&\frac{1}{n}\sum_i (X_i -  \mu_i)^2 +  \frac{1}{n}\sum_i (\mu_i - \frac{1}{n}\sum_j \mu_j)^2  - (\frac{1}{n}\sum_i X_i -\frac{1}{n} \sum_i \mu_i)^2.
\end{align*}

We then have by the definition of $V$ that with probability $1-\delta$
\begin{equation}\label{eq:varia}
 \hat{V} - V = \frac{1}{n}\sum_{i=1}^n (X_i -  \mu_i)^2 -\frac{1}{n} \sum_{i=1}^n \si_i^2  - (\frac{1}{n}\sum_i X_i -\frac{1}{n} \sum_i \mu_i)^2.
\end{equation}

If the assumptions of Lemma~\ref{ss:variance} are verified, we have with probability $1-\delta$
\begin{align*}
 \P \Big( \sum_{i=1}^n (X_i - \mu_i)^2 - \sum_{i=1}^n\si_i^2 \geq n \upsilon \Big) &=  \P \Bigg[ \exp\Big(\lambda(\sum_{i=1}^n |X_i - \mu_i|^2 - \sum_{i=1}^n \si_i^2)\Big) \geq \exp(n\lambda \upsilon) \Bigg]\\
&\leq \E\Bigg[\frac{\exp\Big(\lambda(\sum_{i=1}^n |X_i - \mu_i|^2 - \sum_{i=1}^n\si_i^2)\Big)}{\exp(n\lambda \upsilon)}\Bigg]\\
&\leq \prod_{i=1}^n \E\Bigg[\frac{\exp\Big(\lambda (|X_i - \mu_i|^2 - \si_i^2)\Big)}{\exp(\lambda \upsilon)}\Bigg]\\
&\leq 2\exp(\frac{\lambda^2}{2}\sum_{i=1}^n \frac{\si_i^2}{2(1-\lambda b)} - n\lambda\upsilon).
\end{align*}

If we take $\lambda = \frac{n\upsilon}{\sum_{i=1}^n\sigma_i^2+nb\upsilon}$ we obtain with probability $1-\delta$
\begin{equation}\label{proba:sibgaussvar}
 \P \Big(  \sum_{i=1}^n (X_i -  \mu_i)^2 - \sum_{i=1}^n \si_i^2 \geq  n\upsilon^2 \Big) \leq \exp(-\frac{n^2\upsilon^2}{2(\sum_{i=1}^n\si_i^2+bn\upsilon)}).
\end{equation}

By a union bound we get with probability $1-\delta$ that
\begin{equation*}
 \P \Big(  |\sum_{i=1}^n(X_i - \mu_i)^2 - \sum_{i=1}^n \si_i^2| \geq  n\upsilon \Big) \leq 2\exp(-\frac{n^2\upsilon^2}{2(\sum_{i=1}^n\si_i^2+bn\upsilon)}).
\end{equation*}

This means that with probability $1-\delta$,
\begin{equation}\label{eq:gdev2}
 |\frac{1}{n}\sum_{i=1}^n(X_i - \mu_i)^2 - \frac{1}{n} \sum_{i=1}^n \si_i^2| \leq \sqrt{\frac{2(\frac{1}{n}\sum_{i=1}^n \si_i^2)\log(2/\delta)}{n}} + \frac{b\log(2/\delta)}{n}.
\end{equation}

Finally, by combining Equations \ref{eq:varia} and \ref{eq:gdev2} with Lemma~\ref{lem:bernstein}, we obtain with probability $1-\delta$
\begin{align*}
 |\hat{V} - V| &\leq \frac{4(\frac{1}{n}\sum_{i=1}^n \si_i^2)\log(2/\delta)}{n} + \frac{2b^2\log(2/\delta)^2}{n^2} +\sqrt{\frac{2(\frac{1}{n}\sum_{i=1}^n \si_i^2)\log(2/\delta)}{n}} + \frac{b\log(2/\delta)}{n}\\
&\leq \sqrt{\frac{2(\frac{1}{n}\sum_{i=1}^n \si_i^2)\log(2/\delta)}{n}} + \frac{(3b + 4\frac{1}{n}\sum_{i=1}^n \si_i^2)\log(2/\delta)}{n}\\
&\leq \sqrt{\frac{2V\log(2/\delta)}{n}} + \frac{(3b + 4V)\log(2/\delta)}{n},
\end{align*}
when $n \geq b\log(2/\delta)$ and because $V \geq \frac{1}{n}\sum_{i=1}^n \si_i^2$.

This implies with probability $1-\delta$ that
\begin{align*}
 &V -  \sqrt{\frac{2V\log(2/\delta)}{n}} + \frac{\log(2/\delta)}{2n}  \leq \hat{V} +  \frac{(3b + 4V)\log(2/\delta)}{n} + \frac{\log(2/\delta)}{2n}\\
&\Leftrightarrow \sqrt{V} - \sqrt{\frac{\log(2/\delta)}{2n}}  \leq \sqrt{\hat{V} +  \frac{(1 + 3b + 4V)\log(2/\delta)}{n}} \\
&\Rightarrow \sqrt{V} - \sqrt{\frac{\log(2/\delta)}{2n}}  \leq \sqrt{\hat{V}} +  \sqrt{\frac{(1 + 3b + 4V)\log(2/\delta)}{n}}\\
&\Rightarrow \sqrt{V}  \leq \sqrt{\hat{V}} +  2\sqrt{\frac{(1 + 3b + 4V)\log(2/\delta)}{n}}.
\end{align*}

On the other hand, we have also with probability $1-\delta$
\begin{align*}
 &\hat{V} \leq V+ \sqrt{\frac{2V\log(2/\delta)}{n}} + \frac{(3b + 4V)\log(2/\delta)}{n}\\
&\Rightarrow \sqrt{\hat{V}} \leq \sqrt{V} +  2\sqrt{\frac{(1 + 3b + 4V)\log(2/\delta)}{n}}.
\end{align*}

Finally, we have with probability $1-\delta$
\begin{equation}
|\sqrt{\hat{V}} - \sqrt{V}| \leq  2\sqrt{\frac{(1 + 3b + 4V)\log(2/\delta)}{n}}. 
\end{equation}
\end{proof}

\end{document}